\documentclass[hidelinks,10pt,conference,letterpaper]{IEEEtran}

\IEEEoverridecommandlockouts

\usepackage{etex}

\usepackage{multicol}

\usepackage{comment}
\usepackage{siunitx}
\usepackage{relsize}
\usepackage{ifthen}

\usepackage[caption=false]{subfig}

\usepackage{graphics} 
\usepackage{rotating}
\usepackage{color}
\usepackage{enumerate}
\usepackage[T1]{fontenc}
\usepackage{psfrag}
\usepackage{epsfig} 
\usepackage{booktabs}
\usepackage{graphicx,url}
\usepackage{multirow}
\usepackage{array}
\usepackage{latexsym}
\usepackage{amsfonts}
\usepackage{amsmath}
\usepackage{amssymb}
\usepackage{xstring}
\usepackage{multirow}
\usepackage{xcolor}
\usepackage{prettyref}
\usepackage{flexisym}
\usepackage{bigdelim}
\usepackage{breqn} 
\usepackage{listings}
\let\labelindent\relax
\usepackage{enumitem}
\usepackage{xspace}
\usepackage{bm}
\graphicspath{{./figures/}}
\usepackage{tikz}
\usetikzlibrary{matrix,calc}

\usepackage{lineno}

\usepackage{mdwlist}

\let\itemize\undefined
\makecompactlist{itemize}{stditemize}

\newrefformat{prob}{Problem\,\ref{#1}}
\newrefformat{def}{Definition\,\ref{#1}}
\newrefformat{sec}{Section\,\ref{#1}}
\newrefformat{sub}{Section\,\ref{#1}}
\newrefformat{prop}{Proposition\,\ref{#1}}
\newrefformat{app}{Appendix\,\ref{#1}}
\newrefformat{alg}{Algorithm\,\ref{#1}}
\newrefformat{cor}{Corollary\,\ref{#1}}
\newrefformat{thm}{Theorem\,\ref{#1}}
\newrefformat{lem}{Lemma\,\ref{#1}}
\newrefformat{fig}{Fig.\,\ref{#1}}
\newrefformat{tab}{Table\,\ref{#1}}

\newcommand{\bdmath}{\begin{dmath}}
\newcommand{\edmath}{\end{dmath}}
\newcommand{\beq}{\begin{equation}}
\newcommand{\eeq}{\end{equation}}
\newcommand{\bdm}{\begin{displaymath}}
\newcommand{\edm}{\end{displaymath}}
\newcommand{\bea}{\begin{eqnarray}}
\newcommand{\eea}{\end{eqnarray}}
\newcommand{\beal}{\beq \begin{array}{lll}}
\newcommand{\eeal}{\end{array} \eeq}
\newcommand{\beas}{\begin{eqnarray*}}
\newcommand{\eeas}{\end{eqnarray*}}
\newcommand{\ba}{\begin{array}}
\newcommand{\ea}{\end{array}}
\newcommand{\bit}{\begin{itemize}}
\newcommand{\eit}{\end{itemize}}
\newcommand{\ben}{\begin{enumerate}}
\newcommand{\een}{\end{enumerate}}

\newcommand{\algS}{\widehat{\calS}}
\newcommand{\calA}{{\cal A}}

\newcommand{\calD}{{\cal D}}

\newcommand{\calS}{{\cal S}}

\newcommand{\calV}{{\cal V}}

\newcommand{\hide}[1]{}

\newcommand{\hiddenText}{{\color{gray} hidden text.}}
\newcommand{\hideWithText}[1]{\hiddenText}

\DeclareMathOperator*{\argmin}{arg\,min}

\newcommand{\tran}{^{\mathsf{T}}}

\newcommand{\diag}[1]{\mathrm{diag}\left(#1\right)}
\newcommand{\trace}[1]{\mathrm{tr}\left(#1\right)}

\newcommand{\inv}{^{-1}}

\newcommand{\until}[1]{\{1, 2\dots, #1\}}

\newcommand{\eye}{{\mathbf I}}

\newcommand{\Real}[1]{ { {\mathbb R}^{#1} } }

\newcommand{\att}{^{(t)}}
\newcommand{\at}[1]{^{(#1)}}

\newcommand{\round}[1]{\mathrm{round}\left( #1 \right)}

\newcommand{\LQG}{LQG\xspace}
\newcommand{\myParagraph}[1]{{\bf #1.}\xspace}

\newcommand{\logdet}{\log\det}
\newcommand{\myFigure}[1]{Fig.~\ref{#1}}
\newcommand{\nrRobots}{n}

\newcommand{\trandom}{{\tt random$^*$}\xspace}
\newcommand{\tlogdet}{{\tt logdet}\xspace}
\newcommand{\tslqg}{{\tt s-LQG}\xspace}
\newcommand{\toptimal}{{\tt optimal}\xspace}
\newcommand{\tallSensors}{{\tt allSensors}\xspace}

\newcommand{\of}[2]{_{#1}(#2)}
\newcommand{\initialCovariance}{\Sigma\att{1}{0}}
\newcommand{\allT}{t=1,2,\ldots,T}

\usepackage{float}

\usepackage{hyperref}
\usepackage{graphicx}
\usepackage{epstopdf}
\usepackage{epsfig}

\graphicspath{{./figures/},{../figures/},{../},{./code/}}

\usepackage{xr}
\usepackage{cite}
\usepackage{amsthm}
\newtheoremstyle{mystyle}
  {}
  {}
  {\itshape}
  {}
  {\bfseries}
  {.}
  { }
  {\thmname{#1}\thmnumber{ #2}\thmnote{ (#3)}}
\theoremstyle{mystyle}

\floatname{algorithm}{Algorithm}

\newtheorem{mydef}{Definition}

\newtheorem{mytheorem}{Theorem}
\newtheorem{mylemma}{Lemma}
\newtheorem{myremark}{Remark}
\newtheorem{mycorollary}{Corollary}
\newtheorem{myproposition}{Proposition}

\newtheorem{myproblem}{Problem}

\renewcommand{\at}[1]{_{#1}}
\renewcommand{\att}[2]{_{#1|#2}}
\renewcommand{\trace}[1]{\textrm{tr}\left( #1\right)}

\usepackage{algorithm}
\usepackage{algcompatible}

\usepackage{tikz}

\title{\huge{Sensing-Constrained LQG Control}}

\author{
Vasileios Tzoumas,$^{1,2}$ Luca Carlone,$^{2}$ George J.~Pappas,$^{1}$ Ali Jadbabaie$^{2}$ 
\thanks{$^{1}$The authors are with the Department of Electrical and Systems Engineering, University of Pennsylvania, Philadelphia, PA 19104 USA (email: {\fontsize{8}{8}\selectfont\ttfamily\upshape \{pappagsg, vtzoumas\}@seas.upenn.edu}).}
\thanks{$^{2}$The authors are with the Institute for Data, Systems and Society, and the Laboratory for Information and Decision Systems, Massachusetts Institute of Technology, Cambridge, MA 02139 USA (email: {\fontsize{8}{8}\selectfont\ttfamily\upshape \{jadbabai, lcarlone, vtzoumas\}@mit.edu}).}
\thanks{This work was supported in part by TerraSwarm, one of six centers of
STARnet, a Semiconductor Research Corporation program sponsored by
MARCO and DARPA, and in part by AFOSR Complex Networks Program.
}
}

\let\OLDthebibliography\thebibliography
\renewcommand\thebibliography[1]{
  \OLDthebibliography{#1}
  \setlength{\parskip}{0pt}
  \setlength{\itemsep}{0pt}
}

\begin{document}

\maketitle


\begin{abstract}
Linear-Quadratic-Gaussian (\LQG) control is concerned 
with the design of 
an optimal controller and estimator for linear Gaussian systems with imperfect state information.
Standard \LQG assumes the set of sensor measurements, to be fed to the estimator, to be given.
However, in many problems, arising in networked systems and robotics, 
one may not be able to use all the available sensors, due to power or payload constraints, 
or may be interested in using the smallest subset of sensors that guarantees the attainment 
of a desired control~goal.
In this paper, we introduce the \emph{sensing-constrained \LQG control} problem, in~which one has 
to jointly design sensing, estimation, and control, under given constraints on the resources spent for sensing.
We~focus on the realistic case in which the  sensing strategy 
has to be selected among a finite set of possible sensing modalities. 
While the computation of the optimal sensing strategy is intractable, 
we~present the first scalable algorithm that computes a near-optimal sensing strategy with provable 
sub-optimality guarantees.  
To this end, we~show that a separation principle holds, which allows the design of sensing, estimation, and control policies in isolation.  
We conclude the paper by discussing two applications of sensing-constrained \LQG control, namely,
\emph{sensing-constrained formation control} and 
\emph{resource-constrained robot navigation}.
\end{abstract}

\begin{tikzpicture}[overlay, remember picture]
\path (current page.north east) ++(-4,-0.2) node[below left] {
	This paper has been accepted for publication in the IEEE Transactions of Automatic Control.
};
\end{tikzpicture}
\begin{tikzpicture}[overlay, remember picture]
\path (current page.north east) ++(-4.8,-0.6) node[below left] {
Please cite the paper as:
	V.~Tzoumas, L.~Carlone, George J.~Pappas, A.~Jadbabaie
};
\end{tikzpicture}
\begin{tikzpicture}[overlay, remember picture]
\path (current page.north east) ++(-3.9,-1) node[below left] {
		``LQG Control and Sensing Co-Design", IEEE Transactions of Automatic Control (TAC), 2020.
};
\end{tikzpicture}


\section{Introduction}\label{sec:Intro}

Traditional approaches to control of systems with partially observable state assume the choice of
sensors 
 used to observe 
 the system is given.
The choice of sensors usually results from a preliminary design phase in which an 
expert designer selects a suitable sensor suite that accommodates estimation requirements (e.g., observability, 
desired estimation error) and system constraints (e.g., size, cost).
Modern control applications, from large networked systems to miniaturized robotics systems, 
pose serious limitations to the applicability of this traditional paradigm. 
In large-scale networked systems (e.g., smart grids or robot swarms), in which new nodes are continuously 
added and removed from the network, 
a manual re-design of the sensors becomes cumbersome and expensive, and it is simply not scalable. 
In miniaturized robot systems, while the set of onboard sensors is fixed, 
it may be desirable to selectively activate only a subset of the sensors during different phases of operation, 
 in order to minimize 
 power consumption. 
%
%
In both application scenarios,
one usually has access to a (possibly large) list of potential sensors, 
but, due to resource constraints (e.g., cost, power), can only utilize 
a subset of them. Moreover, the need for online and large-scale sensor selection demands for 
 automated approaches that efficiently select a subset of sensors 
 to maximize system performance.
%
%
%
%

Motivated by these applications,
in this paper we consider the problem of jointly designing control, estimation, and sensor selection for a 
system with partially observable state.

\myParagraph{Related work}
One body of related work is \emph{control over band-limited communication channels}, which 
investigates the trade-offs between communication constraints (e.g., data rate, quantization, delays)
and control performance (e.g., stability) in networked control systems. 
Early work  provides results on the impact of quantization~\cite{Elia01tac-limitedInfoControl}, finite data 
rates~\cite{Nair04sicon-rateConstrainedControl,Tatikonda04tac-limitedCommControl}, and separation principles for \LQG design with 
communication constraints~\cite{Borkar97-limitedCommControl}; more recent work focuses on privacy constraints~\cite{LeNy14tac-limitedCommControl}. 
We refer the reader to the surveys~\cite{Nair07ieee-rateConstrainedControl,Hespanha07ieee-networkedControl,Baillieul07ieee-networkedControl}.
A~second set of related work is \emph{sensor selection and scheduling}, in which one has to select a (possibly time-varying) 
set of sensors in order to monitor a phenomenon of interest. 
Related literature includes approaches based on randomized sensor selection~\cite{gupta2006stochastic}, dual volume sampling~\cite{avron2013faster,li2017dual},
convex relaxations~\cite{joshi2009sensor,leny2011kalman}, and submodularity~\cite{shamaiah2010greedy,jawaid2015submodularity,tzoumas2015sensor}. 
The third set of related works is \emph{information-constrained (or information-regularized)} \LQG 
control~\cite{Shafieepoorfard13cdc-attentionLQG,tanaka2015sdp}.
Shafieepoorfard and Raginsky~\cite{Shafieepoorfard13cdc-attentionLQG} 
study rationally inattentive control laws for \LQG control and discuss their effectiveness  in stabilizing the system. 
Tanaka and Mitter~\cite{tanaka2015sdp} consider the co-design of sensing, control, and estimation, propose to augment the standard \LQG cost 
with an information-theoretic regularizer, and derive an elegant solution based on semidefinite programming.
The main difference between our proposal and~\cite{tanaka2015sdp} is that we consider the case in which 
the choice of sensors, rather than being arbitrary, is restricted to a finite set of available sensors.


\myParagraph{Contributions}
We extend the Linear-Quadratic-Gaussian (\LQG) control  to the case in which, besides 
designing an optimal controller and estimator, one has to select a set of sensors 
to be used to observe the system state. 
In particular, we formulate the \emph{sensing-constrained} (finite-horizon) \LQG  problem as
the joint design of an optimal control and estimation policy, as well as the
selection of a subset of $k$ out of $N$ available sensors, that minimize the \LQG objective, which quantifies tracking performance and control effort.
We first leverage a separation principle to show that the design of sensing, control, and estimation, can be 
performed independently.  While the computation of the optimal sensing strategy is combinatorial in nature,
a key contribution of this paper is to provide the first scalable algorithm that computes a near-optimal sensing strategy with provable 
sub-optimality guarantees. 
We motivate the importance of the sensing-constrained \LQG problem, and demonstrate the effectiveness of the proposed algorithm in numerical experiments, by considering  
two application scenarios, namely,  
\emph{sensing-constrained formation control} and 
\emph{resource-constrained robot navigation}, which, due to page limitations, we include in the full version of this paper, located at the authors' websites.
All proofs can be found also in the {full version of this paper, located at the authors' websites.}

\myParagraph{Notation} 
Lowercase letters denote vectors and scalars, and uppercase letters denote matrices. We use calligraphic fonts to denote sets.
The identity matrix of size~$n$ is denoted with~$\eye_n$ (dimension is omitted when clear from the context). 
For a matrix $M$ and a vector $v$ of appropriate dimension, we define $\|v\|_M^2 \triangleq v\tran M v$.  For matrices $M_1,M_2,\ldots,M_k$, we define $\diag{M_1,M_2,\ldots,M_k}$ as the block diagonal matrix
with diagonal blocks the $M_1,M_2,\ldots,M_k$.


\section{Sensing-Constrained LQG Control} 
\label{sec:problemStatement}

In this section we formalize the sensing-constrained LQG control problem considered in this paper. 
We start by introducing the notions of  \emph{system},  \emph{sensors}, and  \emph{control policies}.
\paragraph{System} We consider a standard discrete-time (possibly time-varying) 
linear system with additive Gaussian noise:  
\begin{equation}
x\at{t+1}=A_tx_t+B_tu_t + w_t, \quad t = 1, 2, \ldots, T,\label{eq:system}
\end{equation}
where $x_t \in \mathbb{R}^{n_t}$ represents the state of the system at time~$t$, $u_t \in \mathbb{R}^{m_t}$ represents the control action, $w_t$ represents the process noise, and 
$T$ is a finite time horizon. In~addition, we~consider the system's initial condition $x_1$ to be a Gaussian random variable with covariance~$\Sigma\att{1}{0}$, and $w_t$ to be a Gaussian random variable with mean zero and covariance $W_t$, such that $w_t$ is independent of $x_1$ and $w_{t'}$ for all $t'=1,2,\ldots,T$, $t'\neq t$.

\paragraph{Sensors} We consider the case where we have a (potentially large) set of available sensors, which take noisy linear observations of the system's state. 
In particular, let~$\calV$  be a set of indices such that each index $i\in\calV$ uniquely identifies a 
sensor that can be used to observe the state of the system. We consider sensors of the form
\begin{equation}
y_{i,t}=C_{i,t}x_t+v_{i,t}, \quad i \in \calV,\label{eq:sensors}
\end{equation}
where $y_{i,t}\in \mathbb{R}^{p_{i,t}}$ represents the measurement of sensor $i$ at time~$t$, and $v_{i,t}$ represents the measurement noise of sensor~$i$.  We assume $v_{i,t}$ to be a Gaussian random variable with mean zero and positive definite covariance $V_{i,t}$, such that~$v_{i,t}$ is independent of $x_1$, and of $w_{t'}$ for any $t'\neq t$, and independent of $v_{i',t'}$ for all $t'\neq t$, and any $i'\in\calV$, $i' \neq i$.

In this paper we are interested in the case 
in which 
we cannot use all the available sensors, and as a result, we need to
select a convenient subset of sensors in $\calV$ to maximize our control performance 
(formalized in Problem~\ref{prob:LQG} below). 

\begin{mydef}[Active sensor set and measurement model]\label{notation:sensor_noise_matrices}
Given a set of available sensors $\calV$, we say that 
$\calS \subset \calV$ is an \emph{active sensor set} if we can observe 
the measurements from each sensor $i\in \calS$ for all $t = 1,2,\ldots, T$.
Given an active sensor set $\calS=\{i_1, i_2 \ldots, i_{|\calS|}\}$,
we define the following quantities 
\beal
y_{t}(\calS) & \triangleq& [y_{i_1,t}\tran, y_{i_2,t}\tran, \ldots, y_{i_{|\calS|},t}\tran]\tran,\\
C_t(\calS) & \triangleq& [C\at{i_1,t}\tran, C\at{i_2,t}\tran, \ldots, C\at{i_{|\calS|},t}\tran]\tran, \\
V_t(\calS)& \triangleq& \text{diag}[V\at{i_1,t}, V\at{i_2,t}, \ldots,V\at{i_{|\calS|},t}]
\eeal
which lead to the definition of the \emph{measurement model}:
\begin{equation}
y_{t}(\calS) = C_t(\calS) x_t + v_{t}(\calS)\label{eq:activeSensors}
\end{equation}
where $v_t(\calS)$ is a zero-mean Gaussian noise with covariance~$V_t(\calS)$.
Despite the availability of a possibly large set of sensors $\calV$, our observer 
will only have access to the measurements produced by the active sensors.
\end{mydef}

The following paragraph formalizes how 
the choice of the active sensors affects the control policies.

\paragraph{Control policies}  We consider control policies~$u_t$ for all $t=1,2,\ldots,T$ 
that are only informed by the measurements collected by the active sensors:
\begin{equation*}
u_t = u\of{t}{\calS}= u_t(y\of{1}{\calS}, y\of{2}{\calS}, \ldots, y\of{t}{\calS}), \quad t = 1,2,\ldots,T.
\end{equation*}
Such policies are called \emph{admissible}.  

In this paper, we want to find a small set of active sensors $\calS$, 
and admissible controllers $u\of{1}{\calS}, u\of{2}{\calS},\ldots, u\of{T}{\calS}$, to solve the following sensing-constrained \LQG control problem.

\begin{myproblem}[Sensing-constrained LQG control]\label{prob:LQG}
Find a sensor set $\calS \subset \calV$ of cardinality at most $k$ to be active across all times $t=1,2,\ldots, T$, and control policies $u\of{1:T}{\calS} \triangleq \{u\of{1}{\calS},$ $u\of{2}{\calS}, \ldots, u\of{T}{\calS}\}$, that minimize the \LQG cost function:
\begin{equation}
\label{eq:sensingConstrainedLQG}
\min_{\scriptsize \begin{array}{c}
\calS \subseteq \calV, |\calS|\leq k,\\
u\of{1:T}{\calS} 
\end{array}} \sum_{t=1}^{T}\mathbb{E}\left[\|x\of{t+1}{\calS}\|^2_{Q_t} +\|u\of{t}{\calS}\|^2_{R_t}\right],
\end{equation}
where the state-cost matrices $Q_1, Q_2, \ldots, Q_T$ are positive semi-definite, the control-cost matrices $R_1, R_2, \ldots, R_T$ are positive definite, 
and the expectation is taken with respect to the initial condition $x_1$, the process noises $w_1, w_2, \ldots, w_T$, and the measurement noises $v\of{1}{\calS}, v\of{2}{\calS}, \ldots, v\of{T}{\calS}$.
\end{myproblem}

Problem~\ref{prob:LQG} generalizes the imperfect state-information \LQG control problem from the case where all sensors in $\calV$ are active, and only optimal control policies are to be found~\cite[Chapter~5]{bertsekas2005dynamic}, to the case where only a few sensors in $\calV$ can be active, and both optimal sensors and control policies are to be found jointly. 
While we already noticed that admissible control policies depend on the active sensor set $\calS$, it is worth noticing that 
this in turn implies that the state evolution also depends on $\calS$; for this reason we write $x\of{t+1}{\calS}$ in eq.~\eqref{eq:sensingConstrainedLQG}.
The intertwining between control and sensing calls for a joint design strategy.
In~the following section we focus on the design of a jointly optimal control {and sensing solution to Problem~\ref{prob:LQG}.}



\section{Joint Sensing and Control Design}

In this section we first present a separation principle that decouples sensing, estimation, and control, 
and allows designing them in cascade (Section~\ref{sec:separability}).
We then present a scalable algorithm for sensing and control design (Section~\ref{sec:designAlgorithm}).  

\subsection{Separability of Optimal Sensing and Control Design}\label{sec:separability}

We characterize the jointly optimal control and sensing solutions to Problem~\ref{prob:LQG}, and prove that  
they can be found in two separate steps, where first the sensing design is computed, and second the {corresponding optimal control design is found.}

\begin{mytheorem}[Separability of optimal sensing and control design]\label{th:LQG_closed}
Let the sensor set $\calS^\star$ and the controllers $u_1^\star, u_2^\star, \ldots, u_T^\star$ be a solution to the sensing-constrained \LQG Problem~\ref{prob:LQG}. Then, $\calS^\star$  and $u_1^\star,$ $u_2^\star,\ldots,u_T^\star$ can be computed in cascade as follows: 
\begin{align}
\calS^\star &\in \argmin_{\calS \subseteq \calV, |\calS|\leq k}\sum_{t=1}^T\text{tr}[\Theta_t\Sigma\att{t}{t}(\calS)],\label{eq:opt_sensors}\\
u_t^\star&=K_t\hat{x}_{t,\calS^\star}, \quad t=1,\ldots,T\label{eq:opt_control}
\end{align}
where $\hat{x}\of{t}{\calS}$ is the Kalman estimator of the state $x_t$, i.e.,
$\hat{x}\of{t}{\calS}\triangleq\mathbb{E}(x_t|y\of{1}{\calS},y\of{2}{\calS},\ldots,y\of{t}{\calS})$,
and $\Sigma\att{t}{t}(\calS)$ is $\hat{x}\of{t}{\calS}$'s error covariance, i.e.,
$\Sigma\att{t}{t}(\calS)\triangleq\mathbb{E}[(\hat{x}\of{t}{\calS}-x_t)(\hat{x}\of{t}{\calS}-x_t)\tran]$~\cite[Appendix~E]{bertsekas2005dynamic}.
In~addition, the matrices $\Theta_t$ and $K_t$ are independent of the selected sensor set $\calS$, and  they are computed as follows: the matrices $\Theta_t$ and $K_t$ are the solution of the backward Riccati recursion
\beal\label{eq:control_riccati}
S_t &= Q_t + N_{t+1}, \\
N_t &= A_t\tran (S_t\inv+B_tR_t\inv B_t\tran)\inv A_t, \\
M_t &= B_t\tran S_t B_t + R_t, \\ 
K_t &= - M_t\inv B_t\tran S_t A_t, \\ 
\Theta_t &= K_t\tran M_t K_t,
\eeal
with boundary condition $N_{T+1}=0$.
\end{mytheorem}

\begin{myremark}[Certainty equivalence principle]
The control gain matrices 
$K_1, K_2, \ldots, K_T$
are the same as the ones that make the controllers $(K_1x_1$, $K_1x_2,\ldots, K_T x_T)$ optimal for the perfect state-information version of Problem~\ref{prob:LQG}, where the state $x_t$ is known to the controllers~\cite[Chapter~4]{bertsekas2005dynamic}.
\end{myremark}


\begin{algorithm}[t]
\caption{\mbox{Joint Sensing and Control design \hspace{-.25mm}for \hspace{-.25mm}Problem~\ref{prob:LQG}.}}
\begin{algorithmic}[1]
\REQUIRE  Time horizon $T$, available sensor set $\calV$, covariance matrix $\initialCovariance$ of initial condition $x_1$; for all $t=1,2,\ldots, T$, system matrix $A_t$, input matrix $B_t$, LQG cost matrices $Q_t$ and $R_t$, process noise covariance matrix $W_t$; and for all sensors $i \in \calV$, measurement matrix $C_{i,t}$, and measurement noise covariance matrix~$V_{i,t}$.
\ENSURE Active sensors $\widehat{\mathcal{S}}$, and control matrices $K_1, \ldots, K_T$.
\STATE{$\widehat{\calS}$ is returned by Algorithm~\ref{alg:greedy} that finds a (possibly approximate) solution to the optimization problem in eq.~\eqref{eq:opt_sensors};
 }
\STATE{$K_1,\ldots,K_T$ are computed using the recursion in eq.~\eqref{eq:control_riccati}.
 }
\end{algorithmic} \label{alg:overall}
\end{algorithm}

Theorem~\ref{th:LQG_closed} decouples the design of the sensing from the controller design.
Moreover, it suggests that once an optimal sensor set~$\calS^\star$ is found, then the optimal controllers are equal to $K_t\hat{x}\of{t}{\calS}$, which correspond to the standard \LQG control policy. 
This should not come as a surprise, since for a given sensing strategy, Problem~\ref{prob:LQG}
reduces to standard \LQG control.

We conclude this section with a remark providing a more intuitive interpretation of the sensor design step in eq.~\eqref{eq:opt_sensors}.

\begin{myremark}[Control-aware sensor design]\label{rmk:interpretation}
In order to provide more insight on the cost function in~\eqref{eq:opt_sensors}, we rewrite it~as: 
\begin{align}
\!\!\!
\displaystyle\sum_{t=1}^{T}\text{tr}[\Theta_t\Sigma\att{t}{t}(\calS)] &\!=\!
\displaystyle\sum_{t=1}^{T}\mathbb{E}\left(\text{tr}\{[x_t-\hat{x}\of{t}{\calS}]\tran\Theta_t[x_t-\hat{x}\of{t}{\calS}]\}\right) \nonumber \\
&\!=\!
\displaystyle\sum_{t=1}^{T}\mathbb{E}\left( \| K_t x_t-K_t \hat{x}\of{t}{\calS}\|^2_{M_t} \right), \label{eq:interpretationRemark}
\end{align}
where in the first line we used the fact that 
$\Sigma\att{t}{t}(\calS) = \mathbb{E}\left[(x_t-\hat{x}\of{t}{\calS})(x_t-\hat{x}\of{t}{\calS})\tran\right]$, 
and in the second line we substituted the definition of $\Theta_t = K_t\tran M_t K_t$ from eq.~\eqref{eq:control_riccati}.

From eq.~\eqref{eq:interpretationRemark}, it is clear that each term 
$\text{tr}[\Theta_t\Sigma\att{t}{t}(\calS)]$ captures the expected control mismatch between the 
imperfect state-information controller $u\of{t}{\calS}=K_t\hat{x}\of{t}{\calS}$ (which is only aware of the 
measurements from the active sensors) and the  perfect state-information controller~$K_tx_t$.
This is an important distinction from the existing sensor selection literature. 
In particular, while standard sensor selection attempts to minimize the estimation covariance, for instance by 
minimizing
\begin{equation}
\sum_{t=1}^{T}\text{tr}[\Sigma\att{t}{t}(\calS)] \triangleq
\displaystyle\sum_{t=1}^{T}\mathbb{E}\left(\| x_t-\hat{x}\of{t}{\calS}\|^2_2 \right),\label{eq:interpretationRemark_kalman}
\end{equation}
the proposed \LQG cost formulation attempts to minimize the estimation error of only the informative states to the perfect state-information controller: for example, the contribution of all $x_t-\hat{x}\of{t}{\calS}$ in the null space of $K_t$ to the total control mismatch in eq.~\eqref{eq:interpretationRemark} is zero.
Hence,  in contrast to minimizing the cost function in eq.~\eqref{eq:interpretationRemark_kalman}, minimizing the cost function in eq.~\eqref{eq:interpretationRemark} results to a control-aware sensing~design.
 
\end{myremark}



\begin{algorithm}[t]
\caption{Sensing design for Problem~\ref{prob:LQG}.}
\begin{algorithmic}[1]
\REQUIRE Time horizon $T$, available sensor set $\calV$, covariance matrix $\initialCovariance$ of system's initial condition $x_1$, and for any time $t =1,2,\ldots, T$, any sensor $i \in \calV$, process noise covariance matrix $W_t$, measurement matrix $C_{i,t}$, and measurement noise covariance matrix $V_{i,t}$.
\ENSURE Sensor set $\algS$.
\STATE Compute $\Theta_1,\Theta_2,\ldots,\Theta_T$ using recursion in eq.~\eqref{eq:control_riccati}; \label{line:computeThetas}
\STATE $\algS\leftarrow\emptyset$; \quad $i \leftarrow 0$; \label{line:initialize}
\WHILE {$i < k$}   \label{line:while}
	\FORALL {$a\in \calV\setminus\algS$} \label{line:startFor1}
	\STATE{$\algS_a\leftarrow \algS\cup \{a\}$; \quad $\initialCovariance(\algS_a)\leftarrow \initialCovariance$;}
	\FORALL {$t = 1,\ldots, T$}
	\vspace{0.3mm}
	\STATE{$\Sigma\att{t}{t}(\algS_a)\leftarrow$}
	\STATE{$[ 
	\Sigma\att{t}{t-1}(\algS_a)\inv + C_t(\algS_a)\tran V_t(\algS_a)\inv C_t(\algS_a)]\inv$;} 
	\STATE{$\Sigma\att{t+1}{t}(\algS_a)\leftarrow A_{t} \Sigma\att{t}{t}(\algS_a) A_{t}\tran + W_{t}$;}
	\ENDFOR
	\STATE{$\text{cost}_a \leftarrow \sum_{t=1}^T\text{tr}[\Theta_t\Sigma\att{t}{t}(\algS_a)]$;} \label{line:cost}
	\ENDFOR \label{line:endFor1}
\STATE{$a_i \leftarrow \arg\min_{a \in \calV\setminus \mathcal{S}} \text{cost}_a$;} \label{line:best_a}
\STATE{$\algS\leftarrow \algS \cup \{a_i\}$; \quad $i \leftarrow i+1$;} \label{line:add_a}
\ENDWHILE
\end{algorithmic} \label{alg:greedy}
\end{algorithm}

\subsection{Scalable Near-optimal Sensing and Control Design}\label{sec:designAlgorithm}

This section proposes a practical design algorithm for Problem~\ref{prob:LQG}.
The pseudo-code of the algorithm is presented in Algorithm~\ref{alg:overall}.
Algorithm~\ref{alg:overall} follows the result of Theorem~\ref{th:LQG_closed}, 
and jointly designs sensing and control by first 
computing an active sensor set (line 1 in Algorithm~\ref{alg:overall}) and then computing the control policy
(line 2 in Algorithm~\ref{alg:overall}).
We discuss each step of the design process in the rest of this section.

\subsubsection{Near-optimal Sensing design}
The optimal sensor design can be computed by solving the optimization problem in eq.~\eqref{eq:opt_sensors}.
The problem is combinatorial in nature, since it requires to select a subset of elements of cardinality $k$ out of all the available sensors that 
induces the smallest cost. 

In this section we propose a greedy algorithm, whose pseudo-code is given in Algorithm~\ref{alg:greedy}, 
that computes a (possibly approximate) solution to the problem in eq.~\eqref{eq:opt_sensors}. 
Our interest towards this greedy algorithm is motivated by the fact that it is scalable 
(in Section~\ref{sec:guarantees} we show that its complexity is linear in the number of available sensors) 
and is provably close to the optimal solution of the problem in eq.~\eqref{eq:opt_sensors} 
\mbox{(we provide suboptimality bounds in Section~\ref{sec:guarantees}).}

Algorithm~\ref{alg:greedy} computes the matrices $\Theta_t$ ($\allT$)
which appear in the cost function in eq.~\eqref{eq:opt_sensors} (line~\ref{line:computeThetas}). Note that these matrices are independent 
on the choice of sensors. 
The~set of active sensors $\algS$ is initialized to the empty set (line~\ref{line:initialize}). 
The~``while loop'' in line~\ref{line:while} will be executed~$k$ times and at each time 
a sensor is greedily added to the set of active sensors~$\algS$. 
In~particular, the ``for loop'' in lines~\ref{line:startFor1}-\ref{line:endFor1} 
 computes the estimation covariance resulting by adding a sensor to the current 
active sensor set and the corresponding cost (line~\ref{line:cost}). 
Finally, the sensor inducing the smallest cost is selected (line~\ref{line:best_a}) 
and added \mbox{to the current set of active sensors (line~\ref{line:add_a}).}



\subsubsection{Control policy design}
The optimal control design is computed as in eq.~\eqref{eq:opt_control}, where
the control policy matrices \mbox{$K_1,K_2,\ldots,K_T$ are obtained 
from the recursion~in~eq.~\eqref{eq:control_riccati}.}


In the following section we characterize the approximation and running-time performance of Algorithm~\ref{alg:overall}.


\section{Performance Guarantees for Joint Sensing and Control Design}\label{sec:guarantees}

We prove that Algorithm~\ref{alg:overall} is the first scalable algorithm for the joint sensing and control design Problem~\ref{prob:LQG}, and that it achieves a value for the \LQG cost function in eq.~\eqref{eq:sensingConstrainedLQG} that is finitely close to the optimal.
We start by introducing the notion of supermodularity ratio (Section~\ref{sec:submodularity}), which will 
enable to bound the sub-optimality gap of Algorithm~\ref{alg:overall} (Section~\ref{sec:performanceGuarantees}).

\subsection{Supermodularity ratio of monotone functions}\label{sec:submodularity}

We define the supermodularity ratio of monotone functions.  We start with the notions of monotonicity and supermodularity.

\begin{mydef}[Monotonicity~{\cite{nemhauser78analysis}}]
Consider any finite ground set~$\mathcal{V}$.  The set function $f:2^\calV\mapsto \mathbb{R}$ is non-increasing if and only if for any $\mathcal{A}\subseteq \mathcal{A}'\subseteq\calV$, $f(\mathcal{A})\geq f(\mathcal{A}')$.
\end{mydef}

\begin{mydef}[Supermodularity~{\cite[Proposition 2.1]{nemhauser78analysis}}]\label{def:sub}
Consider any finite ground set $\calV$.  The set function $f:2^\calV\mapsto \mathbb{R}$ is supermodular if and only if
for any $\mathcal{A}\subseteq \mathcal{A}'\subseteq\calV$ and $x\in \calV$, $f(\mathcal{A})-f(\mathcal{A}\cup \{x\})\geq f(\mathcal{A}')-f(\mathcal{A}'\cup \{x\})$.
\end{mydef}
In words, a set function $f$ is supermodular if and only if it satisfies the following intuitive diminishing returns property: for any $x\in \mathcal{V}$, the marginal drop $f(\mathcal{A})-f(\mathcal{A}\cup \{x\})$ diminishes as $\mathcal{A}$ grows; equivalently, for any $\mathcal{A}\subseteq \mathcal{V}$ and $x\in \mathcal{V}$, the marginal drop $f(\mathcal{A})-f(\mathcal{A}\cup \{x\})$ is non-increasing.

\begin{mydef}[Supermodularity ratio{~\cite[Definition of elemental curvature on p.~5]{wang2016approximation}}]\label{def:super_ratio}
Consider any finite ground set~$\mathcal{V}$, and a non-increasing set \mbox{function $f:2^\calV\mapsto \mathbb{R}$}.  We define the~supermodularity ratio of $f$ as
\begin{equation*}
\gamma_f = \min_{\mathcal{A} \subseteq \mathcal{V}, x, x' \in \mathcal{V}\setminus\calA} \frac{f(\mathcal{A})-f(\mathcal{A}\cup \{x\})}{f(\mathcal{A}\cup\{x'\})-f[(\mathcal{A}\cup\{x'\})\cup \{x\}]}.
\end{equation*}
\end{mydef}
In words, the supermodularity ratio of a monotone set function $f$ measures how far $f$ is from being supermodular.  In particular, per the Definition~\ref{def:super_ratio} of supermodularity ratio, the supermodularity ratio $\gamma_f$ takes values in $[0,1]$, and
\begin{itemize}
\item $\gamma_f= 1$ if and only if $f$ is supermodular, {since if $\gamma_f= 1$}, then Definition~\ref{def:super_ratio} implies 
$ f(\mathcal{A})-f(\mathcal{A}\cup \{x\})\geq f(\mathcal{A}\cup\{x'\})-f[(\mathcal{A}\cup\{x'\})\cup \{x\}]$,
i.e., the drop $f(\mathcal{A})-f(\mathcal{A}\cup \{x\})$ is non-increasing as new elements are added in $\calA$.
\item $\gamma_f < 1$ if and only if $f$ is \emph{approximately supermodular}, in the sense that if $\gamma_f< 1$, then Definition~\ref{def:super_ratio} implies 
$ f(\mathcal{A})- f(\mathcal{A}\cup \{x\})\geq \textstyle\gamma_f \left\{f(\mathcal{A}\cup\{x'\})-f[(\mathcal{A}\cup\{x'\})\cup \{x\}]\right\}$,
i.e., the drop $f(\mathcal{A})-f(\mathcal{A}\cup \{x\})$ is approximately non-increasing as new elements are added in $\calA$; specifically, 
the supermodularity ratio $\gamma_f$ captures how much ones needs to discount the drop $f(\mathcal{A}\cup\{x'\})-f[(\mathcal{A}\cup\{x'\})\cup \{x\}]$, 
such that $f(\mathcal{A})-f(\mathcal{A}\cup \{x\})$ remains greater then, or equal to, $f(\mathcal{A}\cup\{x'\})-f[(\mathcal{A}\cup\{x'\})\cup \{x\}]$.
\end{itemize}

We next use the notion of supermodularity ratio Definition~\ref{def:super_ratio} to quantify the sub-optimality gap of Algorithm~\ref{alg:overall}.

\subsection{Performance Analysis for Algorithm~\ref{alg:overall}}\label{sec:performanceGuarantees}

We quantify Algorithm~\ref{alg:overall}'s running time, as well as, Algorithm~\ref{alg:overall}'s approximation performance, using the notion of supermodularity ratio introduced in Section~\ref{sec:submodularity}.  We conclude the section by showing that for appropriate LQG cost matrices $Q_1, Q_2, \ldots, Q_T$ and $R_1, R_2, \ldots, R_T$, Algorithm~\ref {alg:overall} achieves near-optimal approximate performance. 

\begin{mytheorem}[Performance of Algorithm~\ref{alg:overall}]\label{th:approx_bound} 
For any active sensor set $\calS\subseteq\calV$, and admissible control policies $u\of{1:T}{\calS} \triangleq \{u\of{1}{\calS}, u\of{2}{\calS}, \ldots, u\of{T}{\calS}\}$, 
let $h[\calS, u\of{1:T}{\calS}]$ be Problem~\ref{prob:LQG}'s cost function, i.e.,
\begin{equation*}
h[\calS, u\of{1:T}{\calS}]\triangleq\textstyle \sum_{t=1}^{T}\mathbb{E}(\|x\of{t+1}{\calS}\|^2_{Q_t}+\|u\of{t}{\calS}\|^2_{R_t});
\end{equation*}

Further define the following set-valued function and scalar:
\begin{equation}
\textstyle\hspace{-0.6em}g(\calS)\triangleq\min_{u\of{1:T}{\calS}} h[\calS, u\of{1:T}{\calS}], \label{eq:gS}
\end{equation}
\begin{equation*}
\textstyle\hspace{-0.6em}g^\star\triangleq\min_{ \substack{ 
\calS \subseteq \calV, |\calS|\leq k,\\
u\of{1:T}{\calS} }} h[\calS, u\of{1:T}{\calS}].
\end{equation*}

The following results hold true:
\begin{enumerate}
\item \textit{(Approximation quality)}  Algorithm~\ref{alg:overall} returns an active sensor set $\algS\subset \calV$ of cardinality $k$, 
and gain matrices $K_1,$ $K_2, \ldots,K_T$, such that
the cost $h[\algS, u\of{1:T}{\algS}]$ attained by the sensor set $\algS$ and the corresponding
 control policies $u\of{1:T}{\algS} \triangleq \{K_1\hat{x}\of{1}{\algS},\ldots,K_T\hat{x}\of{T}{\algS}\}$ satisfies
\begin{equation}\label{ineq:approx_bound}
\frac{h(\algS, u\of{1:T}{\algS})-g^\star}{g(\emptyset)-g^\star}\leq 
\exp(-\gamma_g)
\end{equation}
where $\gamma_g$ is the supermodularity ratio of 
$g(\calS)$ in eq.~\eqref{eq:gS}.

\item \textit{(Running time)} Algorithm~\ref{alg:overall} runs in $O(k|\calV|Tn^{2.4})$ time, where $n \triangleq\max_{\allT}(n_t)$ is the maximum system size in eq.~\eqref{eq:system}.

\end{enumerate}
\end{mytheorem}

Theorem~\ref{th:approx_bound} 
ensures that  
Algorithm~\ref {alg:overall} is the first scalable algorithm for the sensing-constrained LQG control Problem~\ref{prob:LQG}.  In~particular, Algorithm~\ref {alg:overall}'s running time $O(k|\calV|Tn^{2.4})$ is linear both in the number of available sensors $|\calV|$, and the sensor set cardinality constraint $k$, as well as, linear in the Kalman filter's running time across the time horizon $\until{T}$.  Specifically, the contribution $n^{2.4}T$ in Algorithm~\ref {alg:overall}'s running time comes from the computational complexity of using the Kalman filter to compute the state estimation error covariances $\Sigma\att{t}{t}$ for each $t=1,2,\ldots,T$~\cite[Appendix~E]{bertsekas2005dynamic}.

Theorem~\ref{th:approx_bound} also guarantees that for non-zero ratio~$\gamma_g$
Algorithm~\ref{alg:overall} achieves a value for Problem~\ref{prob:LQG} that is finitely close to the optimal.
In~particular, the bound in ineq.~\eqref{ineq:approx_bound} improves as  $\gamma_g$ increases, 
since it is decreasing in~$\gamma_g$, and is characterized by the following extreme behaviors:
for $\gamma_g=1$, the bound in ineq.~\eqref{ineq:approx_bound} is $e^{-1}\simeq .37$, which is the 
minimum for any $\gamma_g \in [0,1]$, and hence, the best bound on Algorithm~\ref{alg:overall}'s approximation performance among all $\gamma_g \in [0,1]$ (ideally, the bound in ineq.~\eqref{ineq:approx_bound} would be~$0$ for $\gamma_g=1$, in which case Algorithm~\ref{alg:overall} would be exact, since it would be implied $h(\algS, u\of{1:T}{\algS})=g^\star$; however, even for supermodular functions, 
the best bound one can achieve in \mbox{the worst-case is $e^{-1}$~\cite{feige1998})}; 
for $\gamma_g=0$, ineq.~\eqref{ineq:approx_bound} is uninformative since 
it simplifies to $h(\algS, u\of{1:T}{\algS}) \leq g(\emptyset) = h(\emptyset, u\of{1:T}{\emptyset})$, which~is~trivially~satisfied.\footnote{The inequality $h(\algS, u\of{1:T}{\algS}) \leq h(\emptyset, u\of{1:T}{\emptyset})$ simply states that a 
control policy that is informed by the active sensor set $\calS$ has better performance than a policy that does not use any sensor;
for a more formal proof we refer the reader to Appendix~B.}

In the remaining of the section, we first prove that if the strict inequality  $\sum_{t=1}^T\Theta_t\succ 0$ holds, where each $\Theta_t$ is defined as in eq.~\eqref{eq:control_riccati}, then the ratio $\gamma_g$ in ineq.~\eqref{ineq:approx_bound} is non-zero, and as result Algorithm~\ref{alg:overall} achieves a near-optimal approximation performance (Theorem~\ref{th:submod_ratio}).  Then,  
we prove that the strict inequality $\sum_{t=1}^T\Theta_t\succ 0$ holds true in all LQG control problem
instances where a zero controller would result in a suboptimal
behavior of the system and, as a result, LQG control design
(through solving Problem~\ref{prob:LQG}) is necessary
to achieve their desired system performance (Theorem~\ref{th:freq}).


\begin{mytheorem}[Lower bound for supermodularity ratio $\gamma_g$]\label{th:submod_ratio}
Let $\Theta_t$ for all $\allT$ be defined as in eq.~\eqref{eq:control_riccati}, $g(\calS)$ be defined as in eq.~\eqref{eq:gS}, and for any sensor $i \in \calV$, $\bar{C}_{i,t}$ be the 
normalized measurement matrix $V_{i,t}^{-1/2}{C}_{i,t}$.

If $\sum_{t=1}^{T}\Theta_t\succ 0$, the supermodularity ratio $\gamma_g$ is non-zero. In addition, if we consider for simplicity that the Frobenius norm of each $\bar{C}_{i,t}$ is $1$, i.e., $\trace{\bar{C}_{i,t}\bar{C}_{i,t}\tran}=1$, and that $\text{tr}[\Sigma\att{t}{t}(\emptyset)]\leq \lambda_\max^2[\Sigma\att{t}{t}(\emptyset)]$, $\gamma_g$'s lower bound is
\begin{align}\label{ineq:sub_ratio_bound}
\begin{split}
\gamma_g\geq &\frac{\lambda_\min(\sum_{t=1}^T \Theta_t) }{\lambda_\max(\sum_{t=1}^T \Theta_t)}\frac{ \min_{t\in\{1,2,\ldots,T\}}\lambda_\min^2[\Sigma\att{t}{t}(\calV)] }{\max_{t\in\{1,2,\ldots,T\}}\lambda_\max^2[\Sigma\att{t}{t}(\emptyset)]}\\
&\dfrac{1+\min_{i\in\calV, t\in\until{T}}\lambda_\min[\bar{C}_{i} \Sigma\att{t}{t}(\calV) \bar{C}_{i}\tran]}{2+\max_{i\in\calV, t\in\until{T}}\lambda_\max[\bar{C}_{i} \Sigma\att{t}{t}(\emptyset) \bar{C}_{i}\tran]
}.
 \end{split}
\end{align}
\end{mytheorem}

The supermodularity ratio bound in ineq.~\eqref{ineq:sub_ratio_bound} suggests 
two cases under which $\gamma_g$ can increase, and correspondingly, the performance bound of Algorithm~\ref{alg:overall} in eq.~\eqref{ineq:approx_bound} can improve:

\setcounter{paragraph}{0}
\paragraph{Case~1 where $\gamma_g$'s bound in ineq.~\eqref{ineq:sub_ratio_bound} increases} When the fraction $\lambda_\min(\sum_{t=1}^T \Theta_t) /\lambda_\max(\sum_{t=1}^T \Theta_t)$ increases to $1$, then the right-hand-side in ineq.~\eqref{ineq:sub_ratio_bound} increases.  Equivalently, the right-hand-side in ineq.~\eqref{ineq:sub_ratio_bound} increases when on average all the directions $x_t^{(i)}-\hat{x}_{t}^{(i)}$ of the estimation errors $x_t-\hat{x}_{t}=(x_t^{(1)}-\hat{x}_{t}^{(1)},x_t^{(2)}-\hat{x}_{t}^{(2)},\ldots, x_t^{(n_t)}-\hat{x}_{t}^{(n_t)})$ become equally important in selecting the active sensor set. To~see~this, consider for~example that $\lambda_\max( \Theta_t)=\lambda_\min(\Theta_t)=\lambda$; then, the cost function in eq.~\eqref{eq:opt_sensors}  that Algorithm~\ref {alg:overall} minimizes to select the active sensor set becomes
\begin{align*}
\sum_{t=1}^{T}\text{tr}[\Theta_t\Sigma\att{t}{t}(\calS)]&=
\lambda\sum_{t=1}^{T}\mathbb{E}\left[\text{tr}(\|	x_t-\hat{x}_{t}(\calS)\|_2^2)\right]\\
&=\lambda\sum_{t=1}^{T}\sum_{i=1}^{n_t}\mathbb{E}\left[\text{tr}(\|	x_t^{(i)}-\hat{x}_{t}^{(i)}(\calS)|_2^2)\right].
\end{align*}
Overall, 
it is {easier} for Algorithm~\ref {alg:overall} to approximate a solution to Problem~\ref {alg:overall} as the cost function in eq.~\eqref{eq:opt_sensors} becomes the cost function in the standard sensor selection problems where one minimizes the total estimation covariance as in eq.~\eqref{eq:interpretationRemark_kalman}.

\paragraph{Case~2 where $\gamma_g$'s bound in  ineq.~\eqref{ineq:sub_ratio_bound} increases}
When either the numerators of the last two fractions in the right-hand-side of ineq.~\eqref{ineq:sub_ratio_bound} increase or the denominators of the last two fractions in the right-hand-side of ineq.~\eqref{ineq:sub_ratio_bound} decrease, then  the right-hand-side in ineq.~\eqref{ineq:sub_ratio_bound} increases.
In~particular, the numerators of the last two fractions in right-hand-side of ineq.~\eqref{ineq:sub_ratio_bound} capture the estimation quality when all available sensors in $\calV$ are used, via the terms of the form $\lambda_\min[\Sigma\att{t}{t}(\calV)]$ and $\lambda_\min[\bar{C}_{i,t} \Sigma\att{t}{t}(\calV) \bar{C}_{i,t}\tran]$.  Interestingly, this suggests that the right-hand-side of ineq.~\eqref{ineq:sub_ratio_bound} increases when the available sensors in $\calV$ are inefficient in achieving low estimation error, that is, when the terms of the form $\lambda_\min[\Sigma\att{t}{t}(\calV)]$ and $\lambda_\min[\bar{C}_{i,t} \Sigma\att{t}{t}(\calV) \bar{C}_{i,t}\tran]$ increase.  Similarly, the denominators of the last two fractions in right-hand-side of ineq.~\eqref{ineq:sub_ratio_bound} capture the estimation quality when no sensors are used, via the terms of the form $\lambda_\max[\Sigma\att{t}{t}(\emptyset)]$ and $\lambda_\max[\bar{C}_{i,t} \Sigma\att{t}{t}(\emptyset) \bar{C}_{i,t}\tran]$. This suggests that the right-hand-side of ineq.~\eqref{ineq:sub_ratio_bound} increases when the measurement noise increases.

We next give a control-level equivalent condition to Theorem~\ref{th:submod_ratio}'s condition $\sum_{t=1}^{T} \Theta_t\succ 0$ for non-zero ratio $\gamma_g$.

\begin{mytheorem}[Control-level condition for near-optimal sensor selection] \label{th:freq}
Consider the 
\LQG problem where for any time $t=1,2,\ldots,T$, the state $x_t$ is known to each controller $u_t$ and the process noise $w_t$ is zero, i.e., the optimization problem
\begin{equation}\label{pr:perfect_state}
\!\!\textstyle\min_{u_{1:T}}\sum_{t=1}^{T}\left.[\|x\at{t+1}\|^2_{Q_t} +\|u_{t}(x_t)\|^2_{R_t}]\right|_{\Sigma\att{t}{t}=W_t=0}.
\end{equation}

Let $A_t$ to be invertible for all $t=1,2,\ldots,T$; the strict inequality $\sum_{t=1}^{T} \Theta_t\succ 0$ holds if and only if for all non-zero initial conditions~$x_1$,
\begin{align*}
0\notin \textstyle\arg\min_{u_{1:T}}\sum_{t=1}^{T}\left.[\|x\at{t+1}\|^2_{Q_t} +\|u_{t}(x_t)\|^2_{R_t}]\right|_{\Sigma\att{t}{t}=W_t=0}.
\end{align*}
\end{mytheorem}

Theorem~\ref{th:freq} suggests that Theorem~\ref{th:submod_ratio}'s sufficient condition $\sum_{t=1}^{T} \Theta_t\succ 0$ for non-zero ratio $\gamma_g$~holds if and only if for any non-zero initial condition~$x_1$ the all-zeroes control policy $u_{1:T}=(0,0,\ldots,0)$ is suboptimal for the noiseless perfect state-information LQG problem in eq.~\eqref{pr:perfect_state}.  

Overall, Algorithm~\ref{alg:overall} is the first scalable algorithm for Problem~\ref{prob:LQG}, and 
(for the \LQG control problem instances of interest
where a zero controller would result in a suboptimal behavior of the system and, as a result, \LQG control design is necessary to achieve their desired system performance)
it achieves close to optimal approximate performance.

\section{Numerical Experiments}\label{sec:exp}

We consider two application scenarios for the proposed sensing-constrained \LQG control framework:
\emph{sensing-constrained formation control}  and 
\emph{resource-constrained robot navigation}. 
We present a Monte Carlo analysis for both scenarios, which demonstrates that 
(i) the proposed sensor selection strategy is near-optimal, and in particular, the resulting \LQG-cost 
(tracking performance) matches the optimal selection in all tested instances 
for which the optimal selection could be computed via a brute-force approach,
(ii) a more naive selection which attempts to minimize the state estimation covariance~\cite{jawaid2015submodularity}
(rather than the \LQG cost) has degraded \LQG tracking performance, often comparable to a random selection,
(iii) in the considered instances, a clever selection of a small subset of sensors can ensure an \LQG cost that is 
close to the one obtained by using all available sensors, hence providing an effective alternative for control under 
sensing constraints~\cite{carlone2017attention}.

\newcommand{\mpw}{4.5cm}
\begin{figure}[t]
\hspace{-4mm}
\begin{minipage}{\textwidth}
\begin{tabular}{cc}%
\begin{minipage}{3.5cm}%
\centering
\includegraphics[width=1.05\columnwidth]{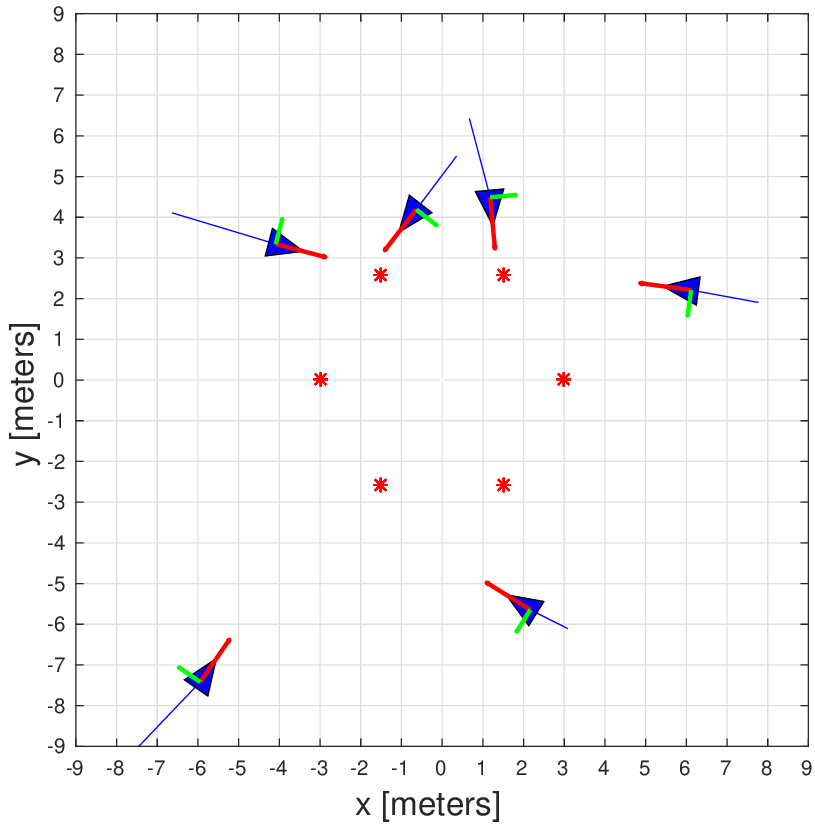} \\
(a) formation control 
\end{minipage}
& \hspace{-3mm}
\begin{minipage}{5.5cm}%
\centering%
\includegraphics[width=1.05\columnwidth,trim=0cm 1cm 0cm 2cm,clip]{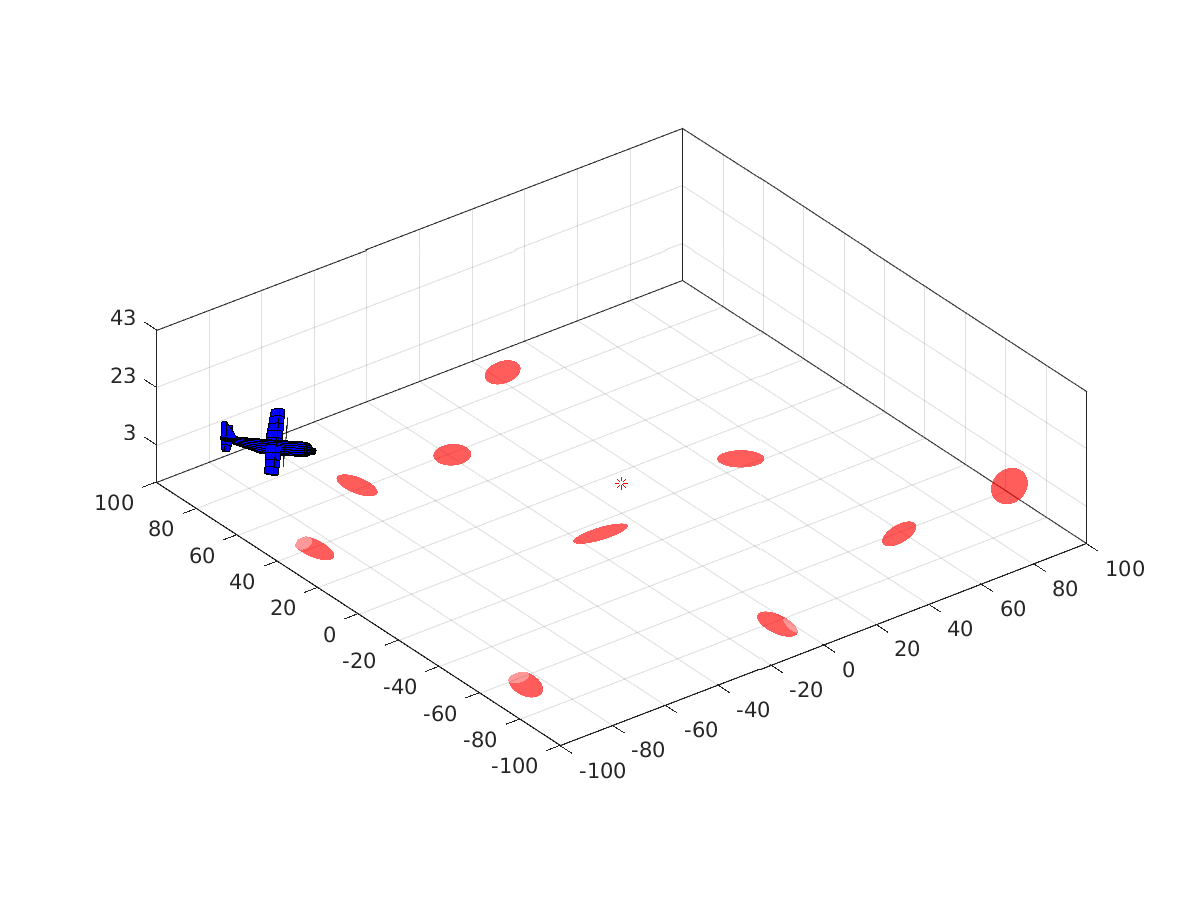} \\
(b) unmanned aerial robot 
\end{minipage}
\end{tabular}
\end{minipage}%
\vspace{-1mm}
\caption{\label{fig:applications}
Examples of applications of the proposed sensing-constrained \LQG control framework:
(a) {sensing-constrained formation control}  and 
(b) {resource-constrained robot navigation}. 
}
\end{figure}

\section{Sensing-constrained formation control}\label{sec:exp-formationControl}

\myParagraph{Simulation setup}  The first application scenario is illustrated in \myFigure{fig:applications}(a).
A team of $\nrRobots$ agents (blue triangles) moves in a 2D scenario.
At time $t=1$, the agents are randomly deployed in a $10\rm{m} \times 10\rm{m}$ square and 
their objective is to reach a target formation shape 
 (red stars); in the example of \myFigure{fig:applications}(a) the desired
 formation has an hexagonal shape, while in general for a formation of $\nrRobots$, the desired 
 formation is an equilateral polygon with $\nrRobots$ vertices. 
 Each robot is modeled as a double-integrator, with state $x_i = [p_i \; v_i]\tran \in \Real{4}$  
 ($p_i$ is the 2D position of agent $i$, while $v_i$ is its velocity), and can control its own acceleration 
 $u_i \in \Real{2}$; the process noise is chosen as a diagonal matrix $W = \diag{[1e^{-2}, \; 1e^{-2}, \; 1e^{-4}, \; 1e^{-4}]}$. 
 Each robot $i$ is equipped with a GPS receiver, which can measure the agent position $p_i$ with a covariance 
 $V_{gps,i} = 2\cdot\eye_2$. Moreover, the agents are equipped with lidar sensors allowing 
 each agent~$i$ to measure the relative position of another agent~$j$ with covariance 
 $V_{lidar,ij} = 0.1\cdot\eye_2$. 
 The agents have very limited on-board resources, hence they can only activate a subset of $k$ sensors.
Hence, the goal is to select the subset of $k$ sensors, as well as to compute the control policy 
that ensure best tracking performance, as measured by the~\LQG~objective.

For our tests, we consider two problem setups. In the first setup, named \emph{homogeneous formation control}, 
the \LQG weigh matrix $Q$ is a block diagonal matrix with $4\times 4$ blocks, 
with each block $i$ chosen as $Q_i = 0.1 \cdot\eye_4$;
since each $4\times 4$ block of~$Q$ weights the tracking error of a robot, in the homogeneous case
the tracking error of all agents is equally important.
In~the second setup, named \emph{heterogeneous formation control}, the matrix $Q$ 
is chose as above, except for one of the agents, say robot 1, for which we choose 
$Q_1 = 10 \cdot\eye_4$; this setup models the case in which 
each agent has a different role or importance, hence one weights differently the 
tracking error of the agents. In~both cases the matrix $R$ is chosen to be the identity matrix. 
The simulation is carried on over $T$ time steps, and $T$ is also chosen as \LQG horizon.
Results are averaged over 100 Monte Carlo runs: at each run we randomize the 
initial estimation covariance {$\initialCovariance$}.

\myParagraph{Compared techniques}
We compare five techniques. All techniques use an \LQG-based estimator and controller, and they 
only differ by the selections of the sensors used.
The~first approach is the optimal sensor selection, denoted as \toptimal, which 
attains the minimum of the cost function in eq.~\eqref{eq:opt_sensors}, and that we compute by enumerating all 
possible subsets;  this brute-force approach is only viable when the number of available sensors is small.
The second approach is a pseudo-random sensor selection, denoted as \trandom, 
 which selects all the GPS measurements and a random subset of the lidar measurements; 
 note that we do not consider a fully random selection since in practice this 
 often leads to an unobservable system, hence causing divergence of the \LQG cost.
The third approach, denoted as \tlogdet, selects sensors so to minimize the average $\logdet$ of the estimation covariance 
over the horizon; this approach resembles~\cite{jawaid2015submodularity} and is agnostic to the control task.
The fourth approach is the proposed sensor selection strategy, described in Algorithm~\ref{alg:greedy}, and is denoted as \tslqg.
Finally, we also report the \LQG performance when all sensors are selected.
This approach is denoted as \tallSensors.

\myParagraph{Results}
The results of our numerical analysis are reported in \myFigure{fig:formationControlStats}.
When not specified otherwise, we consider a formation of $\nrRobots = 4$ agents, 
which can only use a total of $k=6$ sensors, and a control horizon $T=20$. 
\myFigure{fig:formationControlStats}(a) shows the \LQG cost attained by the compared techniques for increasing 
control horizon and for the homogeneous case. 
We note that, in all tested instance, the proposed approach \tslqg matches the optimal selection \toptimal, and both 
approaches are relatively close to \tallSensors, which selects all the 
 available sensors ($\frac{\nrRobots + \nrRobots^2}{2}$). On the other hand \tlogdet leads to worse tracking 
 performance, and it is often close to the pseudo-random selection \trandom.
These considerations are confirmed by the heterogeneous setup, shown in \myFigure{fig:formationControlStats}(b).
In this case the separation between the proposed approach and \tlogdet becomes even larger;
 the intuition here is that the heterogeneous case rewards differently the tracking errors at different agents, 
 hence while \tlogdet attempts to equally reduce the estimation error across the formation, the proposed approach 
 \tslqg selects sensors in a task-oriented fashion, since the matrices $\Theta_t$ for all $\allT$ in the cost function in eq.~\eqref{eq:opt_sensors}
 incorporate the \LQG weight matrices.

\myFigure{fig:formationControlStats}(c) shows the \LQG cost attained by the compared techniques for increasing 
number of selected sensors $k$ and for the homogeneous case.   
We note that for increasing number of sensors all techniques converge to \tallSensors (the entire ground set is selected).
As in the previous case, the proposed approach \tslqg matches the optimal selection \toptimal. 
\myFigure{fig:formationControlStats}(d) shows the same statistics for the heterogeneous case. 
We note that in this case \tlogdet is inferior to \tslqg even in the case with small $k$. 
Moreover, an interesting fact is that \tslqg matches \tallSensors already for $k = 7$, meaning that 
the \LQG performance of the sensing-constraint setup is indistinguishable from the 
one using all sensors; intuitively, in the heterogeneous case, adding more sensors may have 
marginal impact on the \LQG cost (e.g., if the cost rewards a small tracking error for robot 1, it may be 
of little value to take a lidar measurement between robot 3 and 4). This further stresses the importance of the 
proposed framework as a parsimonious way to control a system with minimal resources.

\myFigure{fig:formationControlStats}(e) 
and \myFigure{fig:formationControlStats}(f) show the \LQG cost attained by the compared techniques for increasing 
number of agents, in  the homogeneous and heterogeneous case, respectively.
To ensure observability, we consider $k = \round{1.5n}$, i.e., we select a number of sensors $50\%$ larger than the 
smallest set of sensors that can make the system observable.
We note that \toptimal quickly becomes intractable to compute, hence we omit values beyond $\nrRobots = 4$.
In both figures, the main observation is that the separation among the techniques increases with the number of agents,
since the set of available sensors quickly increases with $n$. Interestingly, in the heterogeneous case  
\tslqg remains relatively close to \tallSensors, implying that for the purpose of \LQG control, using a cleverly selected
 small subset of  sensors still ensures excellent tracking performance.


\newcommand{\resultsFolderFormationControl}{results-formationControl-gpsAndRandom-100-randInitCov-newLQG}
\newcommand{\myhspace}{\hspace{-2mm}}

\renewcommand{\mpw}{4.5cm}
\begin{figure}[t]
\myhspace\myhspace
\begin{minipage}{\textwidth}
\begin{tabular}{cc}%
\myhspace
\begin{minipage}{\mpw}%
\centering
\includegraphics[width=1.05\columnwidth]{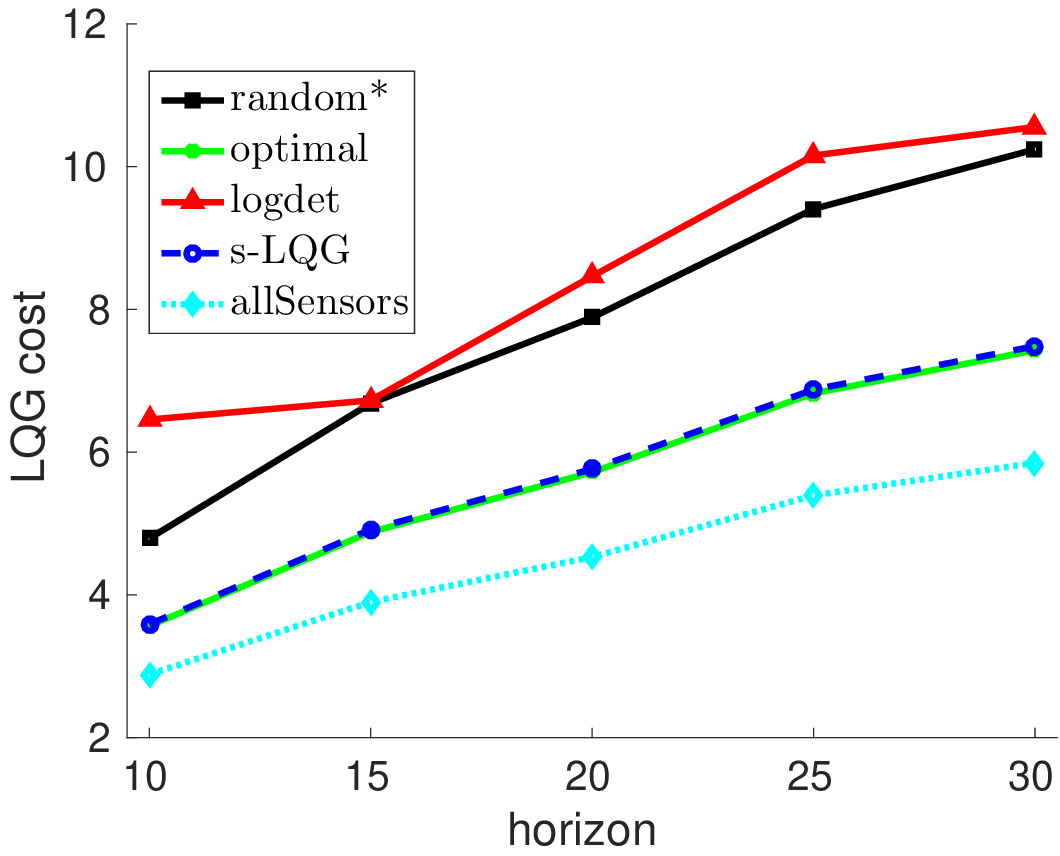} \\
(a) homogeneous 
\end{minipage}
& \myhspace
\begin{minipage}{\mpw}%
\centering%
\includegraphics[width=1.05\columnwidth]{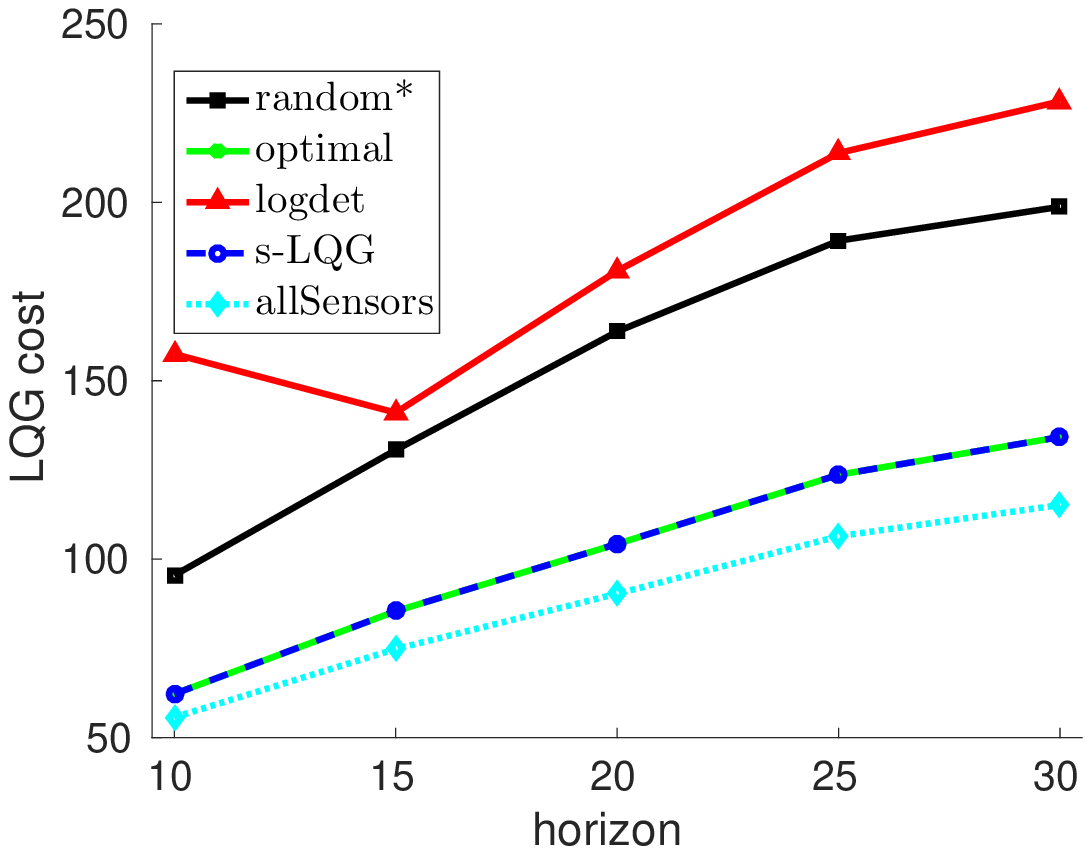} \\
(b) heterogeneous 
\end{minipage}
\\
\myhspace
\begin{minipage}{\mpw}%
\centering
\includegraphics[width=1.05\columnwidth]{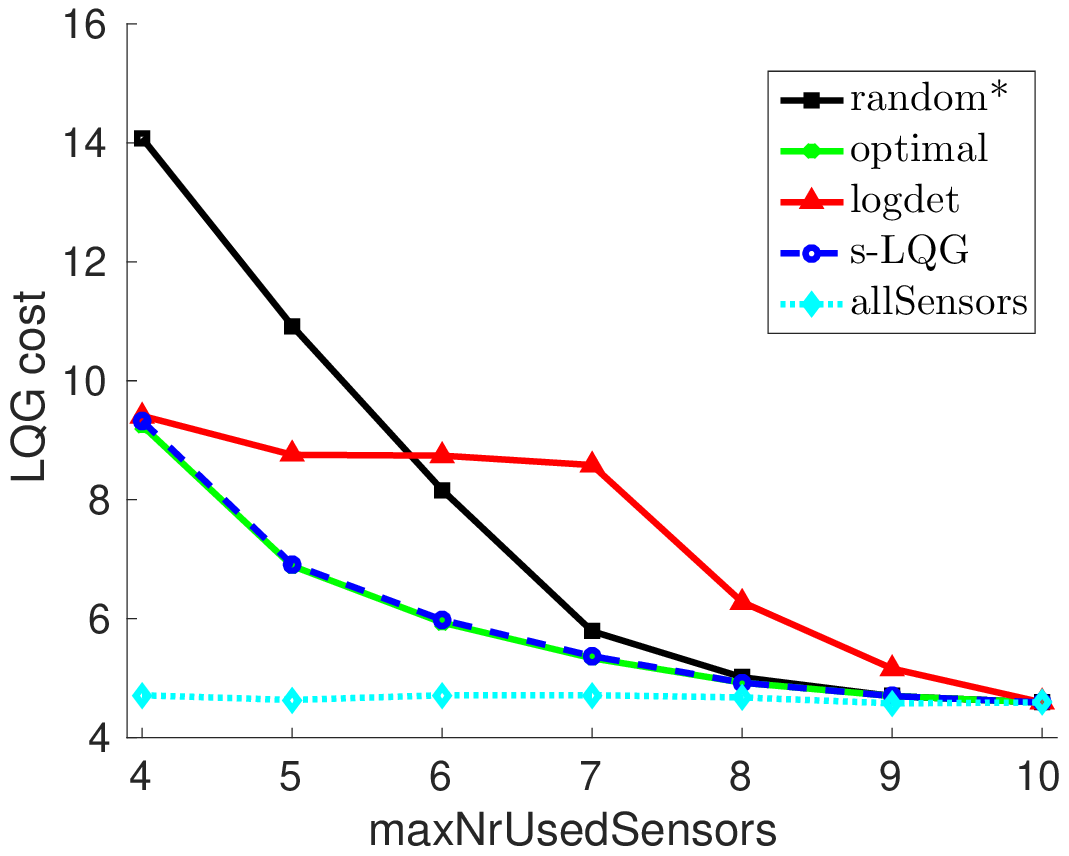} \\
(c) homogeneous 
\end{minipage}
& \myhspace
\begin{minipage}{\mpw}%
\centering%
\includegraphics[width=1.05\columnwidth]{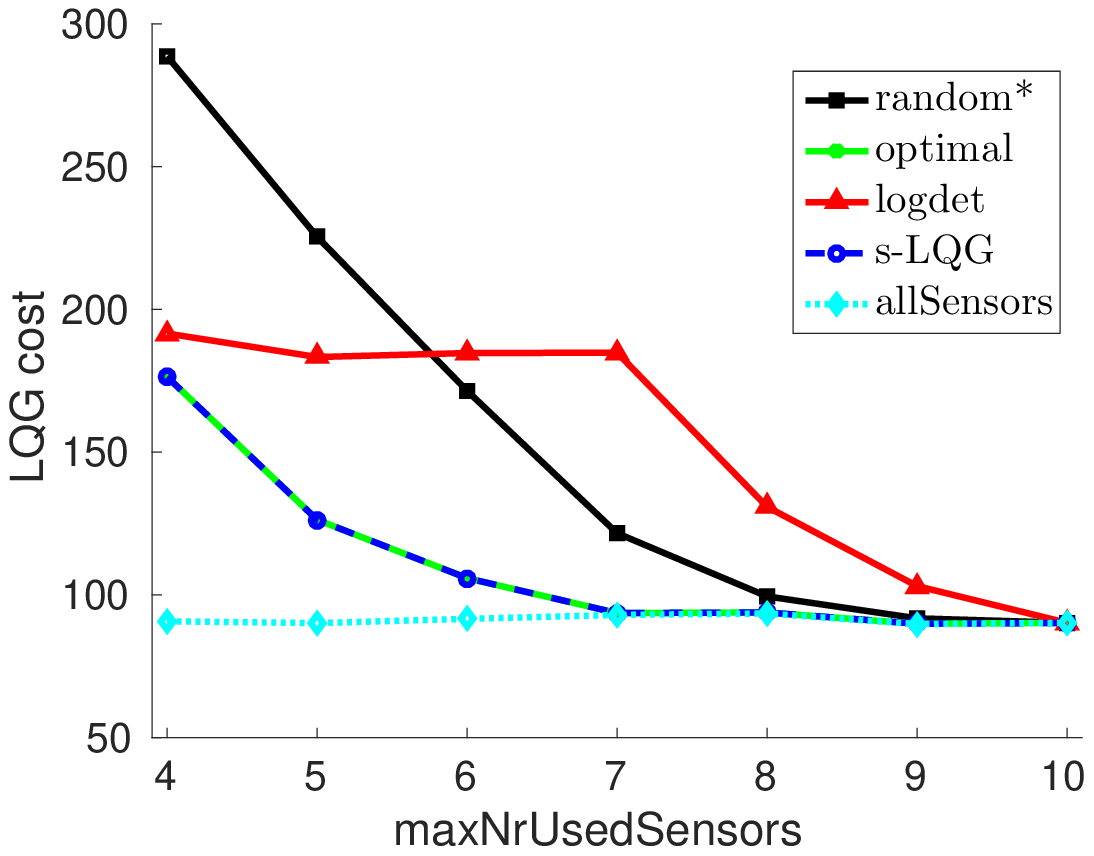} \\
(d) heterogeneous 
\end{minipage}
\\
\myhspace
\begin{minipage}{\mpw}%
\centering
\includegraphics[width=1.05\columnwidth]{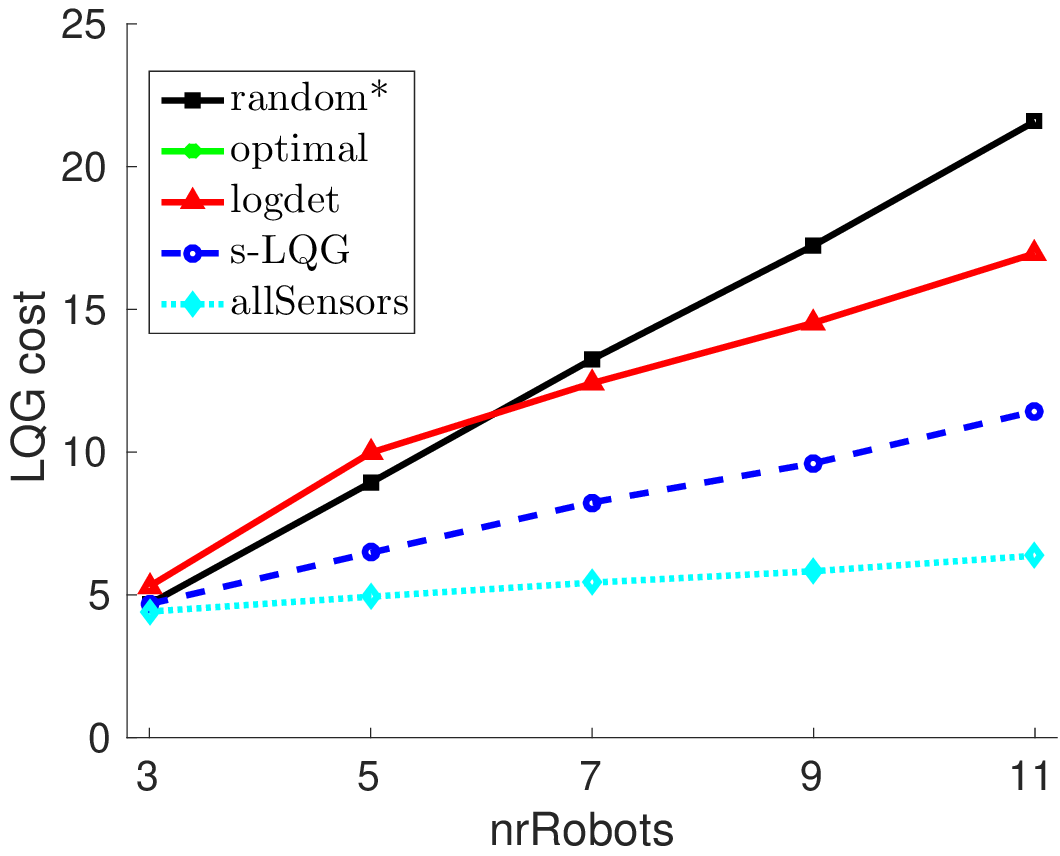} \\
(e) homogeneous 
\end{minipage}
& \myhspace
\begin{minipage}{\mpw}%
\centering%
\includegraphics[width=1.05\columnwidth]{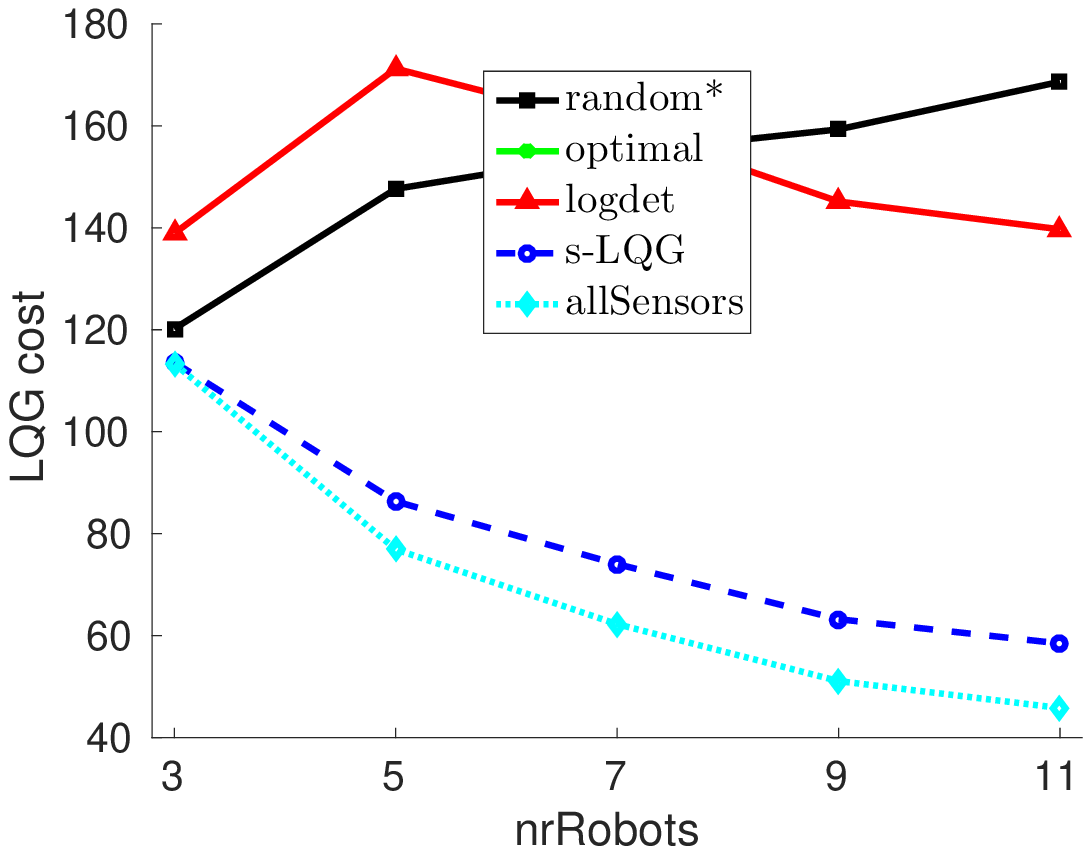} \\
(f) heterogeneous 
\end{minipage}
\end{tabular}
\end{minipage}%
\vspace{-1mm}
\caption{\label{fig:formationControlStats}
\LQG cost for increasing (a)-(b) control horizon $T$, (c)-(d) number of selected 
sensors $k$, and (e)-(f) number of agents $n$.  Statistics are reported for the homogeneous 
formation control setup (left column), and the heterogeneous setup (right column).
Results are averaged over 100 Monte Carlo runs.
}\vspace{-5mm}
\end{figure}

\section{Resource-constrained robot navigation}\label{sec:exp-flyingRobot}

\myParagraph{Simulation setup}  The second application scenario is illustrated in \myFigure{fig:applications}(b).
An unmanned aerial robot (UAV) moves in a 3D scenario, starting from a
randomly selected initial location. 
The objective of the UAV is to land, and more specifically, it has to reach 
the position $[0,\;0,\;0]$ with zero velocity. 
 The UAV is modeled as a double-integrator, with state $x_i = [p_i \; v_i]\tran \in \Real{6}$  
 ($p_i$ is the 3D position of agent $i$, while $v_i$ is its velocity), and can control its own acceleration 
 $u_i \in \Real{3}$; the process noise is chosen as $W = \eye_6$. 
 The UAV is equipped with multiple sensors. It has an on-board GPS receiver, measuring the 
 UAV position  $p_i$ with a covariance $2 \cdot\eye_3$, 
and an altimeter, measuring only the last component of $p_i$ (altitude) with standard deviation $0.5\rm{m}$. 
Moreover, the UAV 
can use a stereo camera to measure the relative position of $\ell$ landmarks on the ground;
for the sake of the numerical example, we assume the location of each landmark to be known 
only approximately, and we associate to each landmark an uncertainty covariance
(red ellipsoids in \myFigure{fig:applications}(b)), which is randomly generated at the 
beginning of each run.
 The UAV has limited on-board resources, hence it can only activate a subset of $k$ sensors.
For instance, the resource-constraints may be due to the power consumption of the GPS and the altimeter,
or may be due to computational constraints that prevent to run multiple object-detection algorithms to detect all landmarks on the ground. 
Similarly to the previous case, we~phrase the problem as a  sensing-constraint \LQG problem, 
and we use $Q = \diag{[1e^{-3},\; 1e^{-3},\;10,\; 1e^{-3},\; 1e^{-3},\; 10]}$ 
and $R = \eye_3$. Note that the structure of $Q$ reflects the fact that during landing 
we are particularly interested in controlling the vertical direction and the vertical velocity 
(entries with larger weight in $Q$), while we are less interested in controlling accurately the 
horizontal position and velocity (assuming a sufficiently large landing site).
In the following, we present results averaged over 100 Monte Carlo runs: in each run, we~randomize the 
covariances describing the landmark position uncertainty.

\myParagraph{Compared techniques}
We consider the five techniques discussed in the previous section.
As in the formation control case, the pseudo-random selection \trandom 
always includes the GPS measurement (which alone ensures observability) 
and a random selection of the other available sensors.

\myParagraph{Results}
The results of our numerical analysis are reported in \myFigure{fig:FlyingRobotStats}.
When not specified otherwise, we consider a total of $k=3$ sensors to be selected, and a control horizon $T=20$. 
\myFigure{fig:FlyingRobotStats}(a) shows the \LQG cost attained by the compared techniques for increasing 
control horizon. For visualization purposes we plot the cost normalized by the horizon, which makes more visible 
the differences among the techniques. Similarly to the formation control example, \tslqg matches the optimal selection \toptimal, 
while \tlogdet and \trandom have suboptimal performance. 

\myFigure{fig:FlyingRobotStats}(b) shows the \LQG cost attained by the compared techniques for increasing 
number of selected sensors $k$. Clearly, all techniques converge to \tallSensors for increasing $k$, but in the 
regime in which few sensors are used \tslqg still outperforms alternative sensor selection schemes, and matches in all cases 
the optimal selection \toptimal.


\newcommand{\resultsFolderFlyingRobot}{results-flyingRobot-gpsAndRandom-100-randInitCov-newLQG}
\renewcommand{\myhspace}{\hspace{-2mm}}

\renewcommand{\mpw}{4.5cm}
\begin{figure}[t]
\myhspace
\begin{minipage}{\textwidth}
\begin{tabular}{cc}%
\myhspace
\begin{minipage}{\mpw}%
\centering
\includegraphics[width=1.05\columnwidth]{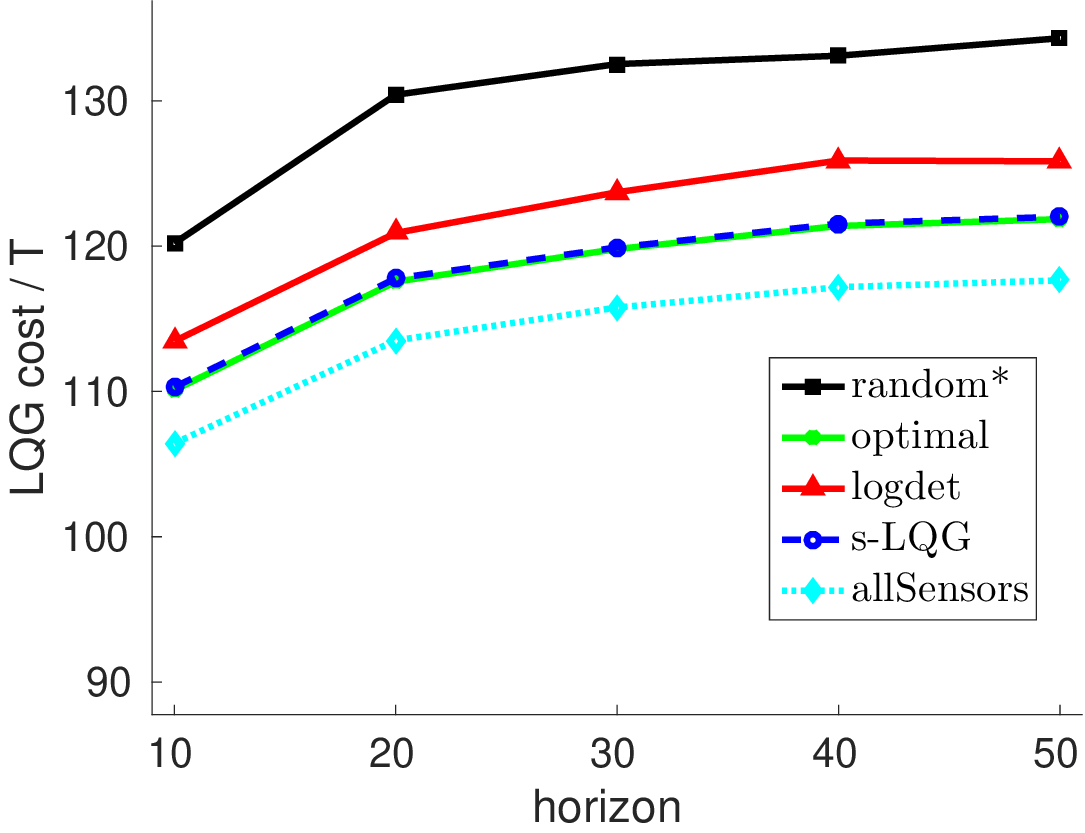} \\
(a) homogeneous 
\end{minipage}
& \myhspace
\begin{minipage}{\mpw}%
\centering%
\includegraphics[width=1.05\columnwidth]{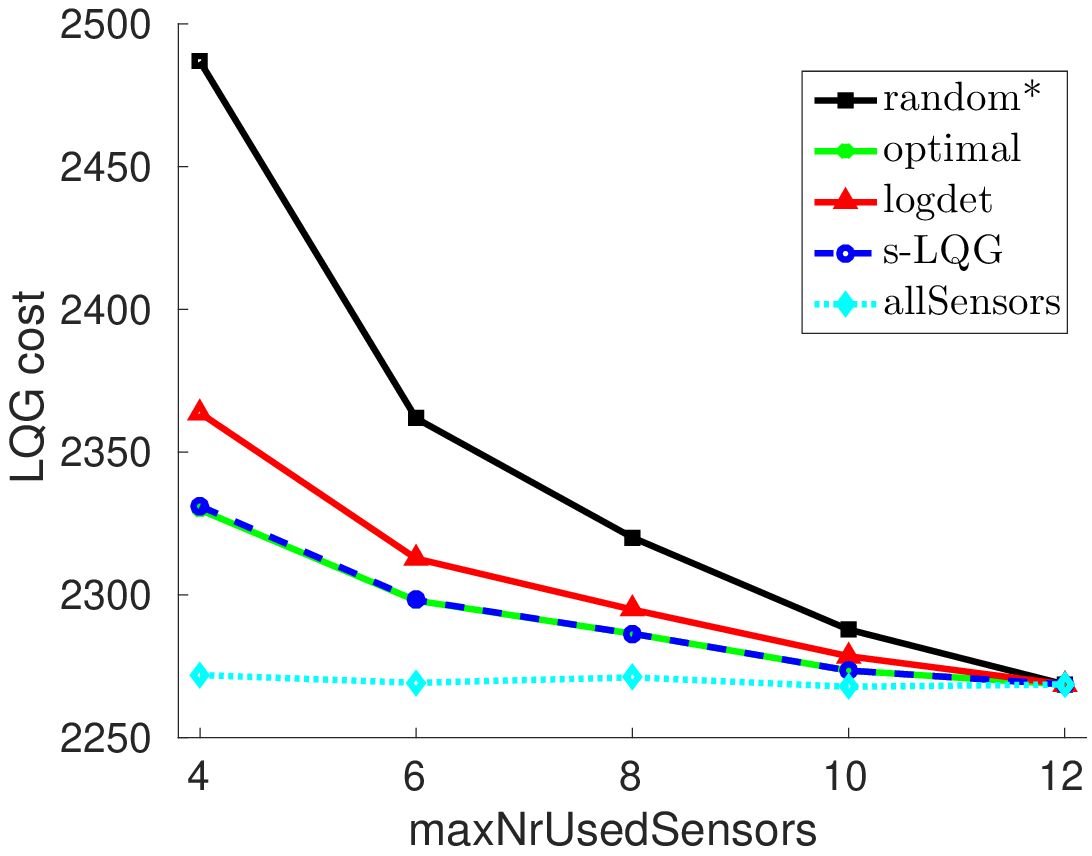} \\
(b) heterogeneous 
\end{minipage}
\end{tabular}
\end{minipage}%
\vspace{-1mm}
\caption{\label{fig:FlyingRobotStats}
\LQG cost for increasing (a) control horizon $T$, and (b) number of selected 
sensors $k$.  Statistics are reported for the heterogeneous setup.
Results are averaged over 100 Monte Carlo runs.
}\vspace{-5mm}
\end{figure}

\section{Concluding Remarks}\label{sec:con}

In this paper, we introduced the \emph{sensing-constrained \LQG control} Problem~\ref{prob:LQG}, 
which is central in modern control applications that range from large-scale networked systems to miniaturized robotics networks.
While the computation of the optimal sensing strategy is intractable, 
We provided the first scalable algorithm for Problem~\ref{prob:LQG}, Algorithm~\ref{alg:overall}, and under mild conditions on the system and LQG matrices, proved that Algorithm~\ref{alg:overall} computes a near-optimal sensing strategy with provable 
sub-optimality guarantees.  
To this end, we~showed that a separation principle holds, which allows the design of sensing, estimation, and control policies in isolation. 
We~motivated the importance of the sensing-constrained \LQG Problem~\ref{prob:LQG}, and 
demonstrated the effectiveness of Algorithm~\ref{alg:overall}, by considering  
two application scenarios:
\emph{sensing-constrained formation control}, and \emph{resource-constrained robot navigation}.


\appendices 


\section*{Appendix A: Preliminary facts}

This appendix contains a set of lemmata that will be used to support the proofs in this paper (Appendices B--F).

\begin{mylemma}[\hspace{-0.05mm}{\cite[Proposition 8.5.5]{bernstein2005matrix}}]\label{lem:inverse}
Consider two positive definite matrices $M_1$ and $M_2$. If $M_1\preceq M_2$ then $M_2\inv\preceq M_1\inv\!\!$.
\end{mylemma}

\begin{mylemma}[Trace inequality~{\cite[Proposition~8.4.13]{bernstein2005matrix}}]\label{lem:trace_low_bound_lambda_min}
Consider a symmetric matrix $A$, and a positive semi-definite matrix $B$ of appropriate dimension. 
Then,
\begin{equation*}
\lambda_\min(A)\trace{B}\leq \trace{AB}\leq \lambda_\max(A)\trace{B}.
\end{equation*}
\end{mylemma}

\begin{mylemma}[Woodbury identity~{\cite[Corollary 2.8.8]{bernstein2005matrix}}]\label{lem:woodbury}
Consider matrices $A$, $C$, $U$ and $V$ of appropriate dimensions, such that $A$, $C$, and $A+UCV$ are invertible. Then,
\begin{equation*}
(A+UCV)\inv=A\inv-A\inv U(C\inv+VA\inv U)\inv VA\inv\!\!.
\end{equation*}
\end{mylemma}

\begin{mylemma}[\hspace{-0.05mm}{\cite[Proposition 8.5.12]{bernstein2005matrix}}]\label{lem:traceAB_mon}
Consider two symmetric matrices $A_1$ and $A_2$, and a positive semi-definite matrix $B$. If~$A_1\preceq A_2$, then $\trace{A_1B}\leq \trace{A_2B}$.  
\end{mylemma}

\begin{mylemma}[\hspace{-0.05mm}{\cite[Appendix~E]{bertsekas2005dynamic}}]\label{lem:covariance_riccati}
For any sensor set $\calS\subseteq\calV$, and for all $\allT$, let $\hat{x}\of{t}{\calS}$ be the Kalman estimator of the state $x_t$, i.e., $\hat{x}\of{t}{\calS}$, and $\Sigma\att{t}{t}(\calS)$ be $\hat{x}\of{t}{\calS}$'s error covariance, i.e.,
$\Sigma\att{t}{t}(\calS)\triangleq\mathbb{E}[(\hat{x}\of{t}{\calS}-x_t)(\hat{x}\of{t}{\calS}-x_t)\tran]$.  Then,  $\Sigma\att{t}{t}(\calS)$ is the solution of the Kalman filtering recursion
\beq\label{eq:covariance_riccati}
\begin{array}{rcl}
\Sigma\att{t}{t}(\calS) \!\!&\!\!=\!\!&\!\![ 
\Sigma\att{t}{t-1}(\calS)\inv + C_t(\calS)\tran V_t(\calS)\inv C_t(\calS)]\inv\!\!,\\
\Sigma\att{t+1}{t}(\calS) \!\!&\!\!=\!\!&\!\! A_{t} \Sigma\att{t}{t}(\calS) A_{t}\tran + W_{t},
\end{array}
\eeq
with boundary condition the $\initialCovariance(\calS)=\initialCovariance$. 
\end{mylemma}

\begin{mylemma}\label{prop:one-step_monotonicity} For any sensor set $\calS\subseteq \calV$, let $\Sigma\att{1}{1}(\calS)$ be defined as in eq.~\eqref{eq:covariance_riccati}, and consider two sensor sets $\calS_1,\calS_2 \subseteq \calV$. {If~$\calS_1\subseteq \calS_2$, then $\Sigma\att{1}{1}(\calS_1)\succeq \Sigma\att{1}{1}(\calS_2)$}.  
\end{mylemma}
\begin{proof}[Proof of Lemma~\ref{prop:one-step_monotonicity}]
Let $\calD=\calS_2\setminus \calS_1$, and observe that for all $\allT$, the notation in Definition~\ref{notation:sensor_noise_matrices} implies
\begin{align}
\hspace{-0.15em} C_t(\calS_2)\tran V_t(\calS_2)\inv C_t(\calS_2)&=\sum_{i\in\calS_2}C_{i,t}\tran V_{i,t}C_{i,t}\nonumber\\ 
&= \sum_{i\in\calS_1}C_{i,t}\tran V_{i,t}C_{i,t}+\sum_{i\in\calD}C_{i,t}\tran V_{i,t}C_{i,t}\nonumber\\
&=\sum_{i\in\calS_1}C_{i,t}\tran V_{i,t}C_{i,t}\nonumber\\
&\succeq
 C_t(\calS_1)\tran V_t(\calS_1)\inv C_t(\calS_1).\label{eq:aux_1}
\end{align}
Therefore, Lemma~\ref{lem:inverse} and ineq.~\eqref{eq:aux_1} imply 
\belowdisplayskip =-12pt
\begin{align*}
&\Sigma\att{1}{1}(\calS_2) =[\Sigma\att{1}{0}\inv+C_1(\calS_2)\tran V_t(\calS_2)\inv C_t(\calS_2)]\inv\preceq\\
& [\Sigma\att{1}{0}\inv+C_1(\calS_1)\tran V_t(\calS_1)\inv C_t(\calS_1)]\inv=\Sigma\att{1}{1}(\calS_1).
\end{align*} 
\end{proof}





\begin{mylemma}\label{cor:from_t_to_t_plus_1}
Let $\Sigma\att{t}{t}$ be defined as in eq.~\eqref{eq:covariance_riccati} with boundary condition the $\Sigma\att{1}{0}$; similarly, let $\bar{\Sigma}\att{t}{t}$ be defined as in eq.~\eqref{eq:covariance_riccati} with boundary condition the $\bar{\Sigma}\att{1}{0}$.
If~$\Sigma\att{t}{t}\preceq\bar{\Sigma}\att{t}{t}$, then  $\Sigma\att{t+1}{t}\preceq\bar{\Sigma}\att{t+1}{t}$.  
\end{mylemma}
\begin{proof}[Proof of Lemma~\ref{cor:from_t_to_t_plus_1}]
We complete the proof in two steps: first, from eq.~\eqref{eq:covariance_riccati}, it its $\Sigma\att{t+1}{t} = A_{t} \Sigma\att{t}{t} A_{t}\tran + W_{t} \preceq A_{t} \bar{\Sigma}\att{t}{t} A_{t}\tran+W_{t}= \bar{\Sigma}\att{t+1}{t} $.  Then, from  $\Sigma\att{t}{t}\preceq\bar{\Sigma}\att{t}{t}$, it follows $A_{t} \Sigma\att{t}{t} A_{t}\tran \preceq A_{t} \bar{\Sigma}\att{t}{t} A_{t}\tran$.
\end{proof}

\begin{mylemma}\label{lem:t_plus_1}
Let $\Sigma\att{t}{t-1}$ be defined as in eq.~\eqref{eq:covariance_riccati} with boundary condition the $\Sigma\att{1}{0}$; similarly, let $\bar{\Sigma}\att{t}{t-1}$ be defined as in eq.~\eqref{eq:covariance_riccati} with boundary condition the $\bar{\Sigma}\att{1}{0}$. If~$\Sigma\att{t}{t-1}\preceq\bar{\Sigma}\att{t}{t-1}$, then $\Sigma\att{t}{t}\preceq\bar{\Sigma}\att{t}{t}$.  
\end{mylemma}
\begin{proof}[Proof of Lemma~\ref{lem:t_plus_1}]
From eq.~\eqref{eq:covariance_riccati}, it is $\Sigma\att{t}{t} =(\Sigma\att{t}{t-1}\inv+C_t\tran V_t\inv C_t)\inv  \preceq (\bar{\Sigma}\att{t}{t-1}\inv+C_t\tran V_t\inv C_t)\inv= \bar{\Sigma}\att{t}{t} $, since Lemma~\ref{lem:inverse} and the condition $\Sigma\att{t}{t-1}\preceq\bar{\Sigma}\att{t}{t-1}$ imply $\Sigma\att{t}{t-1}\inv+C_t\tran V_t\inv C_t\succeq \bar{\Sigma}\att{t}{t-1}\inv+C_t\tran V_t\inv C_t$, which in turn implies $(\Sigma\att{t}{t-1}\inv+C_t\tran V_t\inv C_t)\inv  \preceq (\bar{\Sigma}\att{t}{t-1}\inv+C_t\tran V_t\inv C_t)\inv\!\!$.
\end{proof}

\begin{mycorollary}\label{cor:from_t_to_t}
Let $\Sigma\att{t}{t}$ be defined as in eq.~\eqref{eq:covariance_riccati} with boundary condition the $\Sigma\att{1}{0}$; similarly, let $\bar{\Sigma}\att{t}{t}$ be defined as in eq.~\eqref{eq:covariance_riccati} with boundary condition the $\bar{\Sigma}\att{1}{0}$.
If~$\Sigma\att{t}{t}\preceq\bar{\Sigma}\att{t}{t}$, then $\Sigma\att{t+i}{t+i}\preceq\bar{\Sigma}\att{t+i}{t+i}$ for any positive integer $i$.  
\end{mycorollary}
\begin{proof}[Proof of Corollary~\ref{cor:from_t_to_t}]
If $\Sigma\att{t}{t}\preceq\bar{\Sigma}\att{t}{t}$, from Lemma~\ref{cor:from_t_to_t_plus_1}, we get $\Sigma\att{t+1}{t}\preceq\bar{\Sigma}\att{t+1}{t}$, which, from Lemma~\ref{lem:t_plus_1}, implies  $\Sigma\att{t+1}{t+1}\preceq\bar{\Sigma}\att{t+1}{t+1}$.  By repeating the previous argument another $(i-1)$ times, the proof is complete.
\end{proof}

\begin{mycorollary}\label{cor:from_t_to_t+1}
Let $\Sigma\att{t}{t}$ be defined as in eq.~\eqref{eq:covariance_riccati} with boundary condition the $\Sigma\att{1}{0}$; similarly, let $\bar{\Sigma}\att{t}{t}$ be defined as in eq.~\eqref{eq:covariance_riccati} with boundary condition the $\bar{\Sigma}\att{1}{0}$.
If~$\Sigma\att{t}{t}\preceq\bar{\Sigma}\att{t}{t}$, then $\Sigma\att{t+i}{t+i-1}\preceq\bar{\Sigma}\att{t+i}{t+i-1}$ for any positive integer $i$.  
\end{mycorollary}
\begin{proof}[Proof of Corollary~\ref{cor:from_t_to_t+1}]
If $\Sigma\att{t}{t}\preceq\bar{\Sigma}\att{t}{t}$, from Corollary~\ref{cor:from_t_to_t}, we get $\Sigma\att{t+i-1}{t+i-1}\preceq\bar{\Sigma}\att{t+i-1}{t+i-1}$, which, from Lemma~\ref{cor:from_t_to_t_plus_1}, implies  $\Sigma\att{t+i}{t+i-1}\preceq\bar{\Sigma}\att{t+i}{t+i-1}$.
\end{proof}

\section*{Appendix B: Proof of Theorem~\ref{th:LQG_closed}} 

The proof of Theorem~\ref{th:LQG_closed} follows from the following lemma.

\begin{mylemma}\label{lem:LQG_closed}
For any active sensor set $\calS\subseteq\calV$, and admissible control policies $u\of{1:T}{\calS} \triangleq \{u\of{1}{\calS}, u\of{2}{\calS}, \ldots, u\of{T}{\calS}\}$, 
let $h[\calS, u\of{1:T}{\calS}]$ be Problem~\ref{prob:LQG}'s cost function, i.e.,
\begin{equation*}
h[\calS, u\of{1:T}{\calS}]\triangleq\textstyle \sum_{t=1}^{T}\mathbb{E}(\|x\of{t+1}{\calS}\|^2_{Q_t}+\|u\of{t}{\calS}\|^2_{R_t});
\end{equation*}

Further define the following set-valued function:
\begin{equation*}
\textstyle\ g(\calS)\triangleq\min_{u\of{1:T}{\calS}} h[\calS, u\of{1:T}{\calS}],
\end{equation*}

Consider any sensor set $\calS \subseteq \calV$, and let $u^\star_{1:T,\calS}$  be the vector of control policies $(K_1\hat{x}_{1,\calS}, K_2\hat{x}_{2,\calS}, \ldots, K_T\hat{x}_{T,\calS})$. 
Then $u^\star_{1:T,\calS}$ is an optimal control policy:
\beq
u^\star_{1:T,\calS} \in \argmin_{u_{1:T}(\calS)}h[\calS, u\of{1:T}{\calS}],\label{eq:lem:opt_control}
\eeq
i.e., $g(\calS) =  h[\calS, u^\star\of{1:T}{\calS}]$, and in particular, $u^\star_{1:T,\calS}$ attains a (sensor-dependent) \LQG cost equal to:
\begin{equation}\label{eq:lem:opt_sensors}
g(\calS) = 
\mathbb{E}(\|x_1\|_{N_1})+ \sum_{t=1}^T\left\{\text{tr}[\Theta_t\Sigma\att{t}{t}(\calS)]+\trace{W_tS_t}\right\}.
\end{equation}
\end{mylemma}
\begin{proof}[Proof of Lemma~\ref{lem:LQG_closed}]
Let $h_t[\calS, u_{t:T}(\calS)]$ be the \LQG cost in Problem~\ref{prob:LQG} from time $t$ up to time $T$, i.e.,
\begin{equation*}
h_t[\calS, u_{t:T}(\calS)]\triangleq\sum_{k=t}^T\mathbb{E}(\|x_{k+1}(\calS)\|^2_{Q_t}+\|u_{k}(\calS)\|^2_{R_t}).
\end{equation*}
and define $g_t(\calS)\triangleq\min_{u_{t:T}(\calS)}h_t[\calS, u_{t:T}(\calS)]$.  
Clearly, $g_1(\calS)$ matches the \LQG cost in eq.~\eqref{eq:lem:opt_sensors}.

We complete the proof inductively.  In particular, we first prove Lemma~\ref{lem:LQG_closed} for $t=T$, and then for any other $t\in \{1,2,\ldots, T-1\}$. To this end, we use the following observation: given any sensor set $\calS$, and any time $t \in \{1,2,\ldots, T\}$, \begin{equation}\label{eq:LQG_closed_reduction_step_1}
g_t(\calS)=\min_{u_{t}(\calS)} \left[ \mathbb{E}(\|x_{t+1}(\calS)\|^2_{Q_t}+\|u_t(\calS)\|^2_{R_t})+ g_{t+1}(\calS)\right],
\end{equation}
with boundary condition the $g_{T+1}(\calS)=0$.  
In particular, eq.~\eqref{eq:LQG_closed_reduction_step_1} holds since 
\begin{align*}
g_t(\calS)&=\min_{u_{t}(\calS)}\mathbb{E}\left\{\|x_{t+1}(\calS)\|^2_{Q_t}+\|u_t(\calS)\|^2_{R_t})+\right.\\
&\min_{u_{t+1:T}(\calS)}\left.h_{t+1}[\calS, u_{t+1:T}(\calS)]\right\},
\end{align*} 
where one can easily recognize the second summand to match the definition of
$g_{t+1}(\calS)$. 

We prove Lemma~\ref{lem:LQG_closed} for $t=T$.  From eq.~\eqref{eq:LQG_closed_reduction_step_1}, for $t=T$,
\beal
 g_T(\calS)
\!\!\!&\!\!\!=\min_{u_T(\calS)}\left[ \mathbb{E}(\|x_{T+1}(\calS)\|^2_{Q_T}+\|u_T(\calS)\|^2_{R_T})\right]\\
\!\!\!&\!\!\!=\min_{u_T(\calS)}\left[ \mathbb{E}(\|A_Tx_T+B_Tu_T(\calS)+w_T\|^2_{Q_T}+\right.\\
&\left.\|u_T(\calS)\|^2_{R_T})\right],\label{eq:closed_aux_1}
\eeal
since $x_{T+1}(\calS)=A_Tx_T+B_Tu_T(\calS)+w_T$, as per eq.~\eqref{eq:system}; we note that for notational simplicity we drop henceforth the dependency of $x_T$ on $\calS$ since $x_T$ is independent of $u_T(\calS)$, which is the variable under optimization in the optimization problem in~\eqref{eq:closed_aux_1}. Developing eq.~\eqref{eq:closed_aux_1} we get:
\beal
&&  g_T(\calS)\\
\!\!\!&\!\!\!=\!\!\!&\!\!\!\min_{u_T(\calS)}\left[\mathbb{E}(u_T(\calS)\tran B_T\tran Q_T B_Tu_T(\calS)+w_T\tran Q_Tw_T+\right.\\
&&\left. x_T\tran A_T\tran Q_TA_Tx_T+ 2x_T\tran A_T\tran Q_T B_Tu_T(\calS)+ \right.\\
&&\left. 2x_T\tran A_T\tran Q_T w_T+2 u_T(\calS)\tran B_T\tran Q_Tw_T+\|u_T(\calS)\|^2_{R_T})\right]\\
\!\!\!&\!\!\!=\!\!\!&\!\!\!\min_{u_T(\calS)}\left[ \mathbb{E}(u_T(\calS)\tran B_T\tran Q_T B_Tu_T(\calS)+\|w_T\|_{Q_T}^2+ \right.\\
&&\left.x_T\tran A_T\tran Q_TA_Tx_T+ 2x_T\tran A_T\tran Q_T B_Tu_T(\calS)+\|u_T\|^2_{R_T})\right],\label{eq:closed_aux_2}
\eeal
where the latter equality holds since $w_T$ has zero mean and $w_T$, $x_T$, and $u_T(\calS)$ are independent.  From eq.~\eqref{eq:closed_aux_2}, rearranging the terms, and using the notation in eq.~\eqref{eq:control_riccati},
\beal
&& g_T(\calS) \\
\!\!\!&\!\!\!=\!\!\!&\!\!\!\min_{u_T(\calS)}\left[ \mathbb{E}(u_T(\calS)\tran (B_T\tran Q_T B_T+R_T)u_T(\calS)+\right.\\&&\left. \|w_T\|_{Q_T}^2+x_T\tran A_T\tran Q_TA_Tx_T+ 2x_T\tran A_T\tran Q_T B_Tu_T(\calS)\right]\\
\!\!\!&\!\!\!=\!\!\!&\!\!\!\min_{u_T(\calS)}\left[\mathbb{E}(\|u_T(\calS)\|_{M_T}^2+\|w_T\|_{Q_T}^2+ x_T\tran A_T\tran Q_TA_Tx_T+\right.\\&&\left. 2x_T\tran A_T\tran Q_T B_Tu_T(\calS)\right]\\
\!\!\!&\!\!\!=\!\!\!&\!\!\!\min_{u_T(\calS)}\left[\mathbb{E}(\|u_T(\calS)\|_{M_T}^2+\|w_T\|_{Q_T}^2+ x_T\tran A_T\tran  Q_TA_Tx_T-\right.\\&&\left.2x_T\tran (-A_T\tran Q_T B_TM_T\inv) M_Tu_T(\calS)\right]\\
\!\!\!&\!\!\!=\!\!\!&\!\!\!\min_{u_T(\calS)}\left[ \mathbb{E}(\|u_T(\calS)\|_{M_T}^2+\|w_T\|_{Q_T}^2+ x_T\tran A_T\tran  Q_TA_Tx_T-\right.\\&&\left.2x_T\tran K_T\tran M_Tu_T(\calS)\right]\\
%
\!\!\!&\!\!\!=\!\!\!&\!\!\!\min_{u_T(\calS)}\left[ \mathbb{E}(\|u_T(\calS)-K_Tx_T\|_{M_T}^2+\|w_T\|_{Q_T}^2+\right.\\&&\left. x_T\tran (A_T\tran  Q_TA_T- K_T\tran M_T K_T)x_T\right]\\
\!\!\!&\!\!\!=\!\!\!&\!\!\!\min_{u_T(\calS)}\left(\mathbb{E}(\|u_T(\calS)-K_Tx_T\|_{M_T}^2+\|w_T\|_{Q_T}^2+\right.\\&&\left. x_T\tran (A_T\tran  Q_TA_T- \Theta_T)x_T\right)\\
\!\!\!&\!\!\!=\!\!\!&\!\!\!\min_{u_T(\calS)}\left[ \mathbb{E}(\|u_T(\calS)-K_Tx_T\|_{M_T}^2+\|w_T\|_{Q_T}^2+ \|x_T\|_{N_T}^2\right]\\
\!\!\!&\!\!\!=\!\!\!&\!\!\!\min_{u_T(\calS)} \mathbb{E}(\|u_T(\calS)-K_Tx_T\|_{M_T}^2)+\trace{W_TQ_T}+ \\&& \mathbb{E}( \|x_T\|_{N_T}^2),\label{eq:closed_aux_3}
\eeal
where the latter equality holds since $\mathbb{E}(\|w_T\|_{Q_T}^2)=\mathbb{E}\left[\trace{w_T\tran Q_Tw_T}\right]=\trace{\mathbb{E}(w_T\tran w_T)Q_T}=\trace{W_TQ_T}$.  
Now we note that 
\begin{align}
&\min_{u_T(\calS)} \mathbb{E}(\|u_T(\calS)-K_Tx_T\|_{M_T}^2)\nonumber\\
&=\mathbb{E}(\|K_T\hat{x}_{T}(\calS)-K_Tx_T\|_{M_T}^2) \nonumber\\
&=\trace{\Theta_T \Sigma\att{T}{T}(\calS)}, \label{eq:closed_aux_3b}
\end{align}
since $\hat{x}_{T}(\calS)$ is the Kalman estimator of the state $x_T$, i.e., the minimum mean square estimator of $x_T$, which implies that $K_T\hat{x}_{T}(\calS)$ is the minimum mean square estimator of $K_T{x}_{T}(\calS)$~\cite[Appendix~E]{bertsekas2005dynamic}. Substituting~\eqref{eq:closed_aux_3b} back into eq.~\eqref{eq:closed_aux_3}, 
we get: 
\begin{align*}
g_T(\calS)  = \mathbb{E}( \|x_T\|_{N_T}^2) + \trace{\Theta_T \Sigma\att{T}{T}(\calS)} + \trace{W_TQ_T},
\end{align*}
which proves that Lemma~\ref{lem:LQG_closed} holds for $t=T$. 

We now prove that if Lemma~\ref{lem:LQG_closed} holds for $t=l+1$, it also holds for $t=l$. To this end, assume eq.~\eqref{eq:LQG_closed_reduction_step_1} holds for $t=l+1$.  Using the notation in eq.~\eqref{eq:control_riccati},
\beal
g_l(\calS)
\!\!\!&\!\!\!=\!\!\!&\!\!\!\min_{u_l(\calS)} \left[ \mathbb{E}(\|x_{l+1}(\calS)\|^2_{Q_l}+\|u_l(\calS)\|^2_{R_l})+ g_{l+1}(\calS)\right]\\
\!\!\!&\!\!\!=\!\!\!&\!\!\!\min_{u_l(\calS)} \left\{ \mathbb{E}(\|x_{l+1}(\calS)\|^2_{Q_l}+\|u_l(\calS)\|^2_{R_l})+ \right.\\
&&\left.\mathbb{E}(\|x_{l+1}(\calS)\|_{N_{l+1}}^2)+ \sum_{k=l+1}^T\left[\trace{\Theta_k\Sigma\att{k}{k}(\calS)}+\right.\right.\\
&&\left.\left.\trace{W_kS_k}\right]\right\}\\
\!\!\!&\!\!\!=\!\!\!&\!\!\!\min_{u_l(\calS)} \left\{ \mathbb{E}(\|x_{l+1}(\calS)\|^2_{S_l}+\|u_l(\calS)\|^2_{R_l})+\right.\\
&&\left.\sum_{k=l+1}^T[\trace{\Theta_k\Sigma\att{k}{k}(\calS)}+\trace{W_kS_k}]\right\}\\
\!\!\!&\!\!\!=\!\!\!&\!\!\!\sum_{k=l+1}^T[\trace{\Theta_k\Sigma\att{k}{k}(\calS)}+\trace{W_kS_k}]+\\
&&\min_{u_l(\calS)} \mathbb{E}(\|x_{l+1}(\calS)\|^2_{S_l}+\|u_l(\calS)\|^2_{R_l}).\label{eq:closed_aux_10}
\eeal

In eq.~\eqref{eq:closed_aux_10}, for the last summand in the last right-hand-side, by following the same steps as for the proof of Lemma~\ref{lem:LQG_closed} for $t=T$, we have:
\beal
&\min_{u_l(\calS)} \mathbb{E}(\|x_{l+1}(\calS)\|^2_{S_l}+\|u_l(\calS)\|^2_{R_l})=\\
&\mathbb{E}(\|x_l\|_{N_l}^2)+\trace{\Theta_l\Sigma\att{l}{l}(\calS)}+\trace{W_lQ_l},\label{eq:closed_aux_5}
\eeal
and $u_l(\calS)=K_l\hat{x}_{l}(\calS)$.  Therefore, by substituting eq.~\eqref{eq:closed_aux_5} back to eq.~\eqref{eq:closed_aux_10}, we get:
\beal
g_l(\calS)
\!\!\!&\!\!\!=\!\!\!&\!\!\! \mathbb{E}(\|x_l\|_{N_l}^2)+\sum_{k=l}^T[\trace{\Theta_k\Sigma\att{k}{k}(\calS)}+\trace{W_kS_k}].\label{eq:closed_aux_4}
\eeal
which proves that if Lemma~\ref{lem:LQG_closed} holds for $t=l+1$, it also holds for $t=l$.
By induction, this also proves that Lemma~\ref{lem:LQG_closed} holds for $l=1$, and we already observed that
 $g_1(\calS)$ matches the original \LQG cost in eq.~\eqref{eq:lem:opt_sensors}, hence concluding the proof.
\end{proof}

\begin{proof}[Proof of Theorem~\ref{th:LQG_closed}]
The proof easily follows from Lemma~\ref{lem:LQG_closed}.
Eq.~\eqref{eq:opt_sensors} is a direct consequence of eq.~\eqref{eq:lem:opt_sensors}, since both $\mathbb{E}(x_1\tran N_1 x_1)=\trace{\Sigma\att{1}{1}N_1}$ and $\sum_{t=1}^T\trace{W_tS_t}$ are independent of the choice of the sensor set $\calS$.  Second,~\eqref{eq:opt_control} directly follows from eq.~\eqref{eq:lem:opt_control}.
\end{proof}

%

\section*{Appendix C: Proof of Theorem~\ref{th:approx_bound}}

The following result is used in the proof of Theorem~\ref{th:approx_bound}.

\begin{myproposition}[Monotonicity of cost function in eq.~\eqref{eq:opt_sensors}]\label{prop:monotonicity}
Consider the cost function in eq.~\eqref{eq:opt_sensors},  namely, for any sensor set $\calS\subseteq \calV$ the set function $\sum_{t=1}^{T}\trace{\Theta_t \Sigma\att{t}{t}(\calS)}$.  Then, for any sensor sets such that $\calS_1\subseteq \calS_2 \subseteq \calV$, it holds $\sum_{t=1}^{T}\trace{\Theta_t \Sigma\att{t}{t}(\calS_1)}\geq \sum_{t=1}^{T}\trace{\Theta_t \Sigma\att{t}{t}(\calS_2)}$.  
\end{myproposition}
\begin{proof}
Lemma~\ref{prop:one-step_monotonicity} implies $\Sigma\att{1}{1}(\calS_1) \succeq \Sigma\att{1}{1}(\calS_2)$, and then, Corollary~\ref{cor:from_t_to_t} implies $\Sigma\att{t}{t}(\calS_1) \succeq \Sigma\att{t}{t}(\calS_2)$.  Finally, for any $\allT$, Lemma~\ref{lem:traceAB_mon} implies $\trace{\Theta_t \Sigma\att{t}{t}(\calS_1)}\geq \trace{\Theta_t \Sigma\att{t}{t}(\calS_2)}$, since each $\Theta_t$ is symmetric.
\end{proof}

\begin{proof}[Proof of part~(1) of Theorem~\ref{th:approx_bound} (Algorithm~\ref{alg:greedy}'s approximation quality)] 
Using Proposition~\ref{prop:monotonicity}, and the supermodularity ratio Definition~\ref{def:super_ratio}, the proof of the upper bound $\exp(-\gamma_g)$ in ineq.~\eqref{ineq:approx_bound} follows the same steps as the proof of~\cite[Theorem~1]{chamon2016near}.
\end{proof}

\begin{proof}[Proof of part~(2) of Theorem~\ref{th:approx_bound} (Algorithm~\ref{alg:overall}'s running time)]
We compute Algorithm~\ref{alg:overall}'s running time by adding the running times of Algorithm~\ref{alg:overall}'s lines 1 and 2:

\setcounter{paragraph}{0}
\paragraph{Running time of Algorithm~\ref{alg:overall}'s line~1} Algorithm~\ref{alg:overall}'s line~1 needs $O(k|\calV|Tn^{2.4})$ time.  In~particular, Algorithm~\ref{alg:overall}'s line~2 running time is the running time of Algorithm~\ref{alg:greedy}, whose running time we show next to be $O(k|\calV|Tn^{2.4})$.  To this end, we first compute the running time of Algorithm~\ref{alg:greedy}'s line~1, and then the  running time of Algorithm~\ref{alg:greedy}'s lines~3--15. Algorithm~\ref{alg:greedy}'s line~1 needs $O(n^{2.4})$ time, using the Coppersmith algorithm for both matrix inversion and multiplication~\cite{coppersmith1990matrix}.  Then, Algorithm~\ref{alg:greedy}'s lines 3--15 are repeated $k$ times, due to the ``while loop'' between lines~3 and~15.  We now need to find the running time of Algorithm~\ref{alg:greedy}'s lines 4--14; to this end, we first find the running time of Algorithm~\ref{alg:greedy}'s lines 4--12, and then the running time of Algorithm~\ref{alg:greedy}'s lines 13 and~14.  In more detail, the running time of Algorithm~\ref{alg:greedy}'s lines 4--12 is $O(|\calV|Tn^{2.4})$, since Algorithm~\ref{alg:greedy}'s lines~5--11 are repeated at most $|\calV|$ times and Algorithm~\ref{alg:greedy}'s lines 6--10, as well as line~11 need $O(Tn^{2.4})$ time, using the Coppersmith-Winograd algorithm for both matrix inversion and multiplication~\cite{coppersmith1990matrix}.  Moreover, Algorithm~\ref{alg:greedy}'s lines 13 and~14 need $O[|\calV|\log(|\calV|)]$ time, since line~13 asks for the minimum among at most~$|\calV|$ values of the $\text{cost}_{(\cdot)}$, which takes $O[|\calV|\log(|\calV|)]$ time to be found, using, e.g., the merge sort algorithm.  In sum, Algorithm~\ref{alg:greedy}'s running time is $O[n^{2.4}+k|\calV|Tn^{2.4}+k |\calV|\log(|\calV|)]=O(k|\calV|Tn^{2.4})$.

\paragraph{Running time of Algorithm~\ref{alg:overall}'s line~2} Algorithm~\ref{alg:overall}'s line~2 needs $O(n^{2.4})$ time, using the Coppersmith algorithm for both matrix inversion and multiplication~\cite{coppersmith1990matrix}.

In sum, Algorithm~\ref{alg:overall}'s running time is $O(k|\calV|Tn^{2.4}+n^{2.4})=O(k|\calV|Tn^{2.4})$.
\end{proof}

\section*{Appendix D: Proof of Theorem~\ref{th:submod_ratio}}

\begin{proof}[Proof of Theorem~\ref{th:submod_ratio}]
We complete the proof by first deriving a lower bound for the numerator of the supermodularity ratio $\gamma_g$, and then, by deriving an upper bound for the denominator  of the supermodularity ratio $\gamma_g$.
 
We use the following notation: $c\triangleq\mathbb{E}(x_1\tran N_1 x_1)+\sum_{t=1}^T\trace{W_tS_t}$,
and for any sensor set $\calS \subseteq \calV$, and time $\allT$, $f_t(\calS)\triangleq\trace{\Theta_t \Sigma\att{t}{t}(\calS)}$.  Then, the cost function $g(\calS)$ in eq.~\eqref{eq:gS} is written as  $g(\calS)=c+\sum_{t=1}^T f_t(\calS),$
due to eq.~\eqref{eq:lem:opt_sensors} in Lemma~\ref{lem:LQG_closed}.

\setcounter{paragraph}{0}
\paragraph{Lower bound for the numerator of the supermodularity ratio $\gamma_g$} Per the supermodularity ratio Definition~\ref{def:super_ratio}, the numerator of the submodularity ratio $\gamma_g$ is of the form
\begin{equation}\label{eq:ratio_aux_1}
\sum_{t=1}^{T}[f_t(\calS)-f_t(\calS\cup \{v\})],
\end{equation}
for some sensor set $\calS\subseteq \calV$, and sensor $v \in \calV$; to lower bound the sum in~\eqref{eq:ratio_aux_1}, we lower bound each $f_t(\calS)-f_t(\calS\cup \{v\})$.
To this end, from~eq.~\eqref{eq:covariance_riccati} in Lemma~\ref{lem:covariance_riccati}, observe
\begin{equation*}
\Sigma\att{t}{t}(\calS\cup \{v\})=[\Sigma\att{t}{t-1}\inv(\calS \cup \{v\})+\sum_{i \in \calS \cup \{v\}} \bar{C}_{i,t}\tran \bar{C}_{i,t}]\inv\!\!.
\end{equation*}
Define $\Omega_t=\Sigma\att{t}{t-1}\inv(\calS)+\sum_{i \in \calS}^T \bar{C}_{i,t}\tran \bar{C}_{i,t}$, and $\bar{\Omega}_{t}=\Sigma\att{t}{t-1}\inv(\calS\cup \{v\})+\sum_{i \in \calS}^T \bar{C}_{i,t}\tran \bar{C}_{i,t}$; using the Woodbury identity in Lemma~\ref{lem:woodbury}, 
\begin{align*}
f_t(\calS\cup \{v\})&=\trace{\Theta_t\bar{\Omega}_{t}\inv}-\\
&\hspace{0.4em}\trace{\Theta_t\bar{\Omega}_{t}\inv \bar{C}_{v,t}\tran (I+\bar{C}_{v,t} \bar{\Omega}_{t}\inv \bar{C}_{v,t}\tran)\inv \bar{C}_{v,t} \bar{\Omega}_{t}\inv}.
\end{align*}
Therefore, for any time $t \in \until{T}$,
\begin{align}
& f_t(\calS)-f_t(\calS\cup \{v\})=
\nonumber\\
&\trace{\Theta_t\Omega_{t}\inv}-\trace{\Theta_t\bar{\Omega}_{t}\inv}+\nonumber\\
&\trace{\Theta_t\bar{\Omega}_{t}\inv \bar{C}_{v,t}\tran (I+\bar{C}_{v,t} \bar{\Omega}_{t}\inv \bar{C}_{v,t}\tran)\inv \bar{C}_{v,t} \bar{\Omega}_{t}\inv}\geq\nonumber\\
& \trace{\Theta_t\bar{\Omega}_{t}\inv \bar{C}_{v,t}\tran (I+\bar{C}_{v,t} \bar{\Omega}_{t}\inv \bar{C}_{v,t}\tran)\inv \bar{C}_{v,t} \bar{\Omega}_{t}\inv},\label{ineq:super_ratio_aux_1}
\end{align}
where ineq.~\eqref{ineq:super_ratio_aux_1} holds because $\trace{\Theta_t\Omega_{t}\inv}\geq \trace{\Theta_t\bar{\Omega}_{t}\inv}$.  In~particular, the inequality $\trace{\Theta_t\Omega_{t}\inv}\geq \trace{\Theta_t\bar{\Omega}_{t}\inv}$ is implied as follows:  Lemma~\ref{prop:one-step_monotonicity} implies $\Sigma\att{1}{1}(\calS)\succeq\Sigma\att{1}{1}(\calS\cup\{u\})$.  Then, Corollary~\ref{cor:from_t_to_t+1} implies $ \Sigma\att{t}{t-1}(\calS)\succeq \Sigma\att{t}{t-1}(\calS\cup \{v\})$, and as a result, Lemma~\ref{lem:inverse} implies $\Sigma\att{t}{t-1}(\calS)\inv\preceq\Sigma\att{t}{t-1}(\calS\cup\{u\})\inv\!\!$. Now, $\Sigma\att{t}{t-1}(\calS)\inv\preceq\Sigma\att{t}{t-1}(\calS\cup\{u\})\inv$ and the definition of $\Omega_t$ and~of~$\bar{\Omega}_t$ imply $\Omega_{t}\preceq \bar{\Omega}_{t}$.  Next, Lemma~\ref{lem:inverse} implies  $\Omega_{t}\inv\succeq \bar{\Omega}_{t}\inv$.  As a result, since also $\Theta_t$ is a symmetric matrix, Lem- \mbox{ma~\ref{lem:traceAB_mon} gives the desired inequality $\trace{\Theta_t\Omega_{t}\inv}\geq \trace{\Theta_t\bar{\Omega}_{t}\inv}$.} 

Continuing from the ineq.~\eqref{ineq:super_ratio_aux_1},
\begin{align}
& f_t(\calS)-f_t(\calS\cup \{v\})\geq\nonumber \\
&\trace{\bar{C}_{v,t} \bar{\Omega}_{t}\inv \Theta_t\bar{\Omega}_{t}\inv \bar{C}_{v,t}\tran (I+\bar{C}_{v,t} \bar{\Omega}_{t}\inv \bar{C}_{v,t}\tran)\inv }\geq \nonumber\\
&\lambda_\min((I+\bar{C}_{v,t} \bar{\Omega}_{t}\inv \bar{C}_{v,t}\tran)\inv)\trace{\bar{C}_{v,t} \bar{\Omega}_{t}\inv \Theta_t\bar{\Omega}_{t}\inv \bar{C}_{v,t}\tran },\label{ineq:super_ratio_aux_2}
\end{align}
where ineq.~\eqref{ineq:super_ratio_aux_2} holds due to Lemma~\ref{lem:trace_low_bound_lambda_min}.  From ineq.~\eqref{ineq:super_ratio_aux_2},
\begin{align}
& f_t(\calS)-f_t(\calS\cup \{v\})\geq\nonumber \\
&= \lambda_\max\inv(I+\bar{C}_{v,t} \bar{\Omega}_{t}\inv \bar{C}_{v,t}\tran)\trace{\bar{C}_{v,t} \bar{\Omega}_{t}\inv \Theta_t\bar{\Omega}_{t}\inv \bar{C}_{v,t}\tran }\nonumber\\
&\geq \lambda_\max\inv(I+\bar{C}_{v,t} \Sigma\att{t}{t}(\emptyset) \bar{C}_{v,t}\tran)\trace{\bar{C}_{v,t} \bar{\Omega}_{t}\inv \Theta_t\bar{\Omega}_{t}\inv \bar{C}_{v,t}\tran }\nonumber\\
&= \lambda_\max\inv(I+\bar{C}_{v,t} \Sigma\att{t}{t}(\emptyset) \bar{C}_{v,t}\tran)\trace{\Theta_t \bar{\Omega}_{t}\inv \bar{C}_{v,t}\tran \bar{C}_{v,t}\bar{\Omega}_{t}\inv},\label{ineq:super_ratio_aux_3}
\end{align}
where we used $\bar{\Omega}_{t}\inv\preceq\Sigma\att{t}{t}(\emptyset)$, which holds because of the following: the definition of $\bar{\Omega}_t$ implies  $\bar{\Omega}_{t}\succeq \Sigma\inv\att{t}{t-1}(\calS\cup \{v\})$, and as a result, from Lemma~\ref{lem:inverse} we get $\bar{\Omega}_{t}\inv \preceq \Sigma\att{t}{t-1}(\calS\cup \{v\})$.  In addition, Corollary~\ref{cor:from_t_to_t+1} and the fact that $\Sigma\att{1}{1}(\calS\cup \{v\})\preceq \Sigma\att{1}{1}(\emptyset)$, which holds due to Lemma~\ref{prop:one-step_monotonicity}, imply $\Sigma\att{t}{t-1}(\calS\cup \{v\})\preceq \Sigma\att{t}{t-1}(\emptyset)$.  Finally, from eq.~\eqref{eq:covariance_riccati} in Lemma~\ref{lem:covariance_riccati} it is $\Sigma\att{t}{t-1}(\emptyset)=\Sigma\att{t}{t}(\emptyset)$. Overall, the desired inequality $\bar{\Omega}_{t}\inv\preceq\Sigma\att{t}{t}(\emptyset)$ holds.

Consider a time $t' \in \until{T}$ such that for any time $t \in \{1,2,\ldots,T\}$ it is $\bar{\Omega}_{t'}\inv \bar{C}_{v,t'}\tran \bar{C}_{v,t'}\bar{\Omega}_{t'}\inv  \preceq  \bar{\Omega}_{t}\inv \bar{C}_{v,t}\tran \bar{C}_{v,t} \bar{\Omega}_{t}\inv\!\!$, and let $\Phi$ be the matrix $\bar{\Omega}_{t'}\inv \bar{C}_{v,t'}\tran \bar{C}_{v,t'} \bar{\Omega}_{t'}\inv $; similarly, let $l$  be the $\min_{t\in\until{T},u \in \calV}\lambda_\max\inv(I+\bar{C}_{v,t} \Sigma\att{t}{t}(\emptyset) \bar{C}_{v,t}\tran)$. 
Summing ineq.~\eqref{ineq:super_ratio_aux_3} across all times $t \in \until{T}$, and using Lemmata~\ref{lem:traceAB_mon} and~\ref{lem:trace_low_bound_lambda_min},
\begin{align*}
 g(\calS)-g(\calS\cup\{v\})&\geq  l\sum_{t=1}^T \trace{\Theta_t\bar{\Omega}_{t}\inv \bar{C}_{v,t}\tran \bar{C}_{v,t} \bar{\Omega}_{t}\inv }\\
& \geq l\sum_{t=1}^T \trace{\Theta_t \Phi}\\
& =l\trace{\Phi \sum_{t=1}^T \Theta_t }\\
& \geq l \lambda_\min\left(\sum_{t=1}^T \Theta_t \right)\trace{\Phi}\\
&>0,
\end{align*}
which is non-zero because $\sum_{t=1}^T \Theta_t\succ 0$ and $\Phi$ is a non-zero positive semi-definite matrix. 

Finally, we lower bound $\trace{\Phi}$, using Lemma~\ref{lem:trace_low_bound_lambda_min}:
\begin{align}
\trace{\Phi}&=\trace{\bar{\Omega}_{t'}\inv \bar{C}_{v,t'}\tran \bar{C}_{v,t'} \bar{\Omega}_{t'}\inv }\nonumber\\
&=\trace{ \bar{\Omega}_{t'}^{-2} \bar{C}_{v,t'}\tran\bar{C}_{v,t'}}\nonumber\\
&\geq \lambda_\min(\bar{\Omega}_{t'}^{-2}) \trace{\bar{C}_{v,t'}\tran\bar{C}_{v,t'}}\nonumber\\
&= \lambda_\min^2(\bar{\Omega}_{t'}^{-1}) \trace{\bar{C}_{v,t'}\tran\bar{C}_{v,t'}}\nonumber\\
&\geq \lambda_\min^2(\Sigma\att{t'}{t'}(\calV)) \trace{\bar{C}_{v,t'}\tran\bar{C}_{v,t'}},\label{ineq:super_ratio_aux_10}
\end{align}
where ineq.~\eqref{ineq:super_ratio_aux_10} holds because $\bar{\Omega}_{t'}^{-1}\succeq \Sigma\att{t'}{t'}(\calV)$.  In particular, the inequality $\bar{\Omega}_{t'}^{-1}\succeq \Sigma\att{t'}{t'}(\calS\cup \{v\})$ is derived by applying Lemma~\ref{lem:inverse} to the inequality $\bar{\Omega}_{t'}\preceq \bar{\Omega}_{t'}+\bar{C}_{v,t}\tran\bar{C}_{v,t}\tran=\Sigma\inv\att{t'}{t'}(\calS\cup\{v\})$, where the equality holds by the definition of $\bar{\Omega}_{t'}$.  In~addition, due to Lemma~\ref{prop:one-step_monotonicity} it is  $\Sigma\att{1}{1}(\calS\cup \{v\})\succeq \Sigma\att{1}{1}(\calV)$, and as a result, from Corollary~\ref{cor:from_t_to_t} it also is  $\Sigma\att{t'}{t'}(\calS\cup \{v\})\succeq \Sigma\att{t'}{t'}(\calV)$.  Overall, the desired inequality $\bar{\Omega}_{t'}^{-1}\succeq \Sigma\att{t'}{t'}(\calV)$ holds.

\paragraph{Upper bound for the denominator of the supermodularity ratio $\gamma_g$} The denominator of the submodularity ratio~$\gamma_g$ is of the form
\begin{equation*}
\sum_{t=1}^{T}[f_t(\calS')-f_t(\calS'\cup \{v\})],
\end{equation*}
for some sensor set $\calS'\subseteq \calV$, and sensor $v \in \calV$;  to upper bound it,  from~eq.~\eqref{eq:covariance_riccati} in Lemma~\ref{lem:covariance_riccati} of Appendix~A, observe
\begin{equation*}
\Sigma\att{t}{t}(\calS'\cup \{v\})=[\Sigma\att{t}{t-1}\inv(\calS' \cup \{v\})+\sum_{i \in \calS' \cup \{v\}} \bar{C}_{i,t}\tran \bar{C}_{i,t}]\inv\!\!,
\end{equation*}
and let $H_t=\Sigma\att{t}{t-1}\inv(\calS')+\sum_{i \in \calS'}^T\bar{C}_{i,t}\tran \bar{C}_{i,t}$, and $\bar{H}_{t}=\Sigma\att{t}{t-1}\inv(\calS'\cup \{v\})+\sum_{i \in \calS'}^T \bar{C}_{i,t}\tran \bar{C}_{i,t}$; using the Woodbury identity in Lemma~\ref{lem:woodbury},
\begin{align*}
f_t(\calS'\cup \{v\})&=\trace{\Theta_t\bar{H}_{t}\inv}-\\
&\hspace{0.4em}\trace{\Theta_t\bar{H}_{t}\inv \bar{C}_{v,t}\tran (I+\bar{C}_{v,t} \bar{H}_{t}\inv \bar{C}_{v,t}\tran)\inv \bar{C}_{v,t} \bar{H}_{t}\inv}.
\end{align*}
Therefore, 
\begin{align}
&\sum_{t=1}^T [f_t(\calS')-f_t(\calS'\cup \{v\})]=
\nonumber\\
&\sum_{t=1}^T [\trace{\Theta_t H_{t}\inv}-\trace{\Theta_t\bar{H}_{t}\inv}+\nonumber
\end{align}

\begin{align}
&\trace{\Theta_t\bar{H}_{t}\inv \bar{C}_{v,t}\tran (I+\bar{C}_{v,t} \bar{H}_{t}\inv \bar{C}_{v,t}\tran)\inv \bar{C}_{v,t} \bar{H}_{t}\inv}]\leq\nonumber\\
&\sum_{t=1}^T [\trace{\Theta_t H_{t}\inv}+ \nonumber\\
&\trace{\Theta_t\bar{H}_{t}\inv \bar{C}_{v,t}\tran (I+\bar{C}_{v,t} \bar{H}_{t}\inv \bar{C}_{v,t}\tran)\inv \bar{C}_{v,t} \bar{H}_{t}\inv}],\label{ineq:super_ratio_aux_4}
\end{align}
where ineq.~\eqref{ineq:super_ratio_aux_4} holds since $\trace{\Theta_t\bar{H}_{t}\inv}$ is non-negative.  In eq.~\eqref{ineq:super_ratio_aux_4}, the second term in the sum is upper bounded as follows, using Lemma~\ref{lem:trace_low_bound_lambda_min}:
\begin{align}
&\trace{\Theta_t\bar{H}_{t}\inv \bar{C}_{v,t}\tran (I+\bar{C}_{v,t} \bar{H}_{t}\inv \bar{C}_{v,t}\tran)\inv \bar{C}_{v,t} \bar{H}_{t}\inv}= \nonumber\\
&\trace{\bar{C}_{v,t} \bar{H}_{t}\inv\Theta_t\bar{H}_{t}\inv \bar{C}_{v,t}\tran (I+\bar{C}_{v,t} \bar{H}_{t}\inv \bar{C}_{v,t}\tran)\inv}\leq \nonumber\\
&\trace{\bar{C}_{v,t} \bar{H}_{t}\inv\Theta_t\bar{H}_{t}\inv \bar{C}_{v,t}\tran} \lambda_\max[(I+\bar{C}_{v,t} \bar{H}_{t}\inv \bar{C}_{v,t}\tran)\inv]= \nonumber\\
&\trace{\bar{C}_{v,t} \bar{H}_{t}\inv\Theta_t\bar{H}_{t}\inv \bar{C}_{v,t}\tran} \lambda_\min\inv(I+\bar{C}_{v,t} \bar{H}_{t}\inv \bar{C}_{v,t}\tran)\leq \nonumber\\
&\trace{\bar{C}_{v,t} \bar{H}_{t}\inv\Theta_t\bar{H}_{t}\inv \bar{C}_{v,t}\tran} \lambda_\min\inv(I+\bar{C}_{v,t} \Sigma\att{t}{t}(\calV) \bar{C}_{v,t}\tran),\label{ineq:super_ratio_aux_5}
\end{align}
since $\lambda_\min(I+\bar{C}_{v,t} \bar{H}_{t}\inv \bar{C}_{v,t}\tran)\geq \lambda_\min(I+\bar{C}_{v,t} \Sigma\att{t}{t}(\calV) \bar{C}_{v,t}\tran)$, because  $\bar{H}_{t}\inv\succeq \Sigma\att{t}{t}(\calV)$.  In particular, the inequality $\bar{H}_{t}\inv\succeq \Sigma\att{t}{t}(\calV)$ is derived as follows: first, it is $\bar{H}_{t}\preceq \bar{H}_{t}+\bar{C}_{v,t}\tran \bar{C}_{v,t}=\Sigma\att{t}{t}(\calS'\cup \{v\})\inv\!\!,$ where the equality holds by the definition of $\bar{H}_t$, and now Lemma~\ref{lem:inverse} implies $\bar{H}_{t}\inv\succeq \Sigma\att{t}{t}(\calS'\cup \{v\})$.  In~addition, $\Sigma\att{t}{t}(\calS'\cup\{v\})\succeq \Sigma\att{t}{t}(\calV)$ is implied from Corollary~\ref{cor:from_t_to_t}, since Lemma~\ref{prop:one-step_monotonicity} implies $\Sigma\att{1}{1}(\calS'\cup \{v\})\succeq \Sigma\att{1}{1}(\calV)$.  Overall, the desired inequality $\bar{H}_{t}\inv\succeq \Sigma\att{t}{t}(\calV)$ holds.

Let $l'=\max_{t\in\until{T}, v\in \calV}\lambda_\min\inv(I+\bar{C}_{v,t} \Sigma\att{t}{t}(\calV) \bar{C}_{v,t}\tran)$. 
From ineqs.~\eqref{ineq:super_ratio_aux_4} and~\eqref{ineq:super_ratio_aux_5},
\beal\label{ineq:super_ratio_aux_6}
&\sum_{t=1}^T [f_t(\calS')-f_t(\calS'\cup \{v\})]\leq
\\
&\sum_{t=1}^T [\trace{\Theta_t H_{t}\inv }+ l'\trace{\Theta_t\bar{H}_{t}\inv \bar{C}_{v,t}\tran\bar{C}_{v,t} \bar{H}_{t}\inv}].
\eeal
Consider times $t' \in \until{T}$ and $t'' \in \until{T}$ such that for any time $t \in \{1,2,\ldots,T\}$, it is  $H_{t'}\inv  \succeq H_{t}\inv $ and $\bar{H}_{t''}\inv \bar{C}_{v,t''}\tran\bar{C}_{v,t''} \bar{H}_{t''}\inv \succeq \bar{H}_{t}\inv \bar{C}_{v,t}\tran\bar{C}_{v,t} \bar{H}_{t}\inv$, and let $\Xi=H_{t'}\inv $ and $\Phi'=\bar{H}_{t'}\inv \bar{C}_{v,t'}\tran\bar{C}_{v,t'} \bar{H}_{t'}\inv$.  From ineq.~\eqref{ineq:super_ratio_aux_6}, and Lemma~\ref{lem:traceAB_mon},
\begin{align}
&\sum_{t=1}^T [f_t(\calS')-f_t(\calS'\cup \{v\})]\leq
\nonumber\\
&\sum_{t=1}^T [\trace{\Theta_t  \Xi}+ l'\trace{ \Theta_t  \Phi'}]\leq\nonumber\\
&\trace{\Xi\sum_{t=1}^T\Theta_t }+ l'\trace{\Phi'\sum_{t=1}^T \Theta_t }\leq\nonumber \\
&(\trace{\Xi}+l'\trace{\Phi'})\lambda_\max(\sum_{t=1}^T \Theta_t).\label{ineq:super_ratio_aux_7}
\end{align}

Finally, we upper bound $\trace{\Xi}+l'\trace{\Phi'}$ in ineq.~\eqref{ineq:super_ratio_aux_7}, using Lemma~\ref{lem:trace_low_bound_lambda_min}:
\begin{align}
&\trace{\Xi}+l'\trace{\Phi'}\leq \nonumber\\
&\trace{ H_{t'}\inv}+\\
&l'\lambda_\max^2(\bar{H}_{t''}\inv)\trace{\bar{C}_{v,t''}\tran\bar{C}_{v,t''}}\leq\nonumber\\
&\trace{\Sigma\att{t'}{t'}(\emptyset)}+l'\lambda_\max^2(\Sigma\att{t''}{t''}(\emptyset))\trace{\bar{C}_{v,t''}\tran\bar{C}_{v,t''}},\label{ineq:super_ratio_aux_8}
\end{align}
where ineq.~\eqref{ineq:super_ratio_aux_8} holds because $H_{t'}\inv\preceq \Sigma\att{t'}{t'}(\emptyset)$, and similarly, $\bar{H}_{t''}\inv\preceq \Sigma\att{t''}{t''}(\emptyset)$. In particular, the inequality $H_{t'}\inv\preceq \Sigma\att{t'}{t'}(\emptyset)$ is implied as follows: first, by the definition of $H_{t'}$, it is $H_{t'}\inv=\Sigma\att{t'}{t'}(\calS')$; and finally, Corollary~\ref{cor:from_t_to_t} and the fact that $\Sigma\att{1}{1}(\calS')\preceq \Sigma\att{1}{1}(\emptyset)$, which holds due to Lemma~\ref{prop:one-step_monotonicity}, imply $\Sigma\att{t'}{t'}(\calS')\preceq \Sigma\att{t'}{t'}(\emptyset)$.  In addition, the inequality $\bar{H}_{t''}\inv\preceq \Sigma\att{t''}{t''}(\emptyset)$ is implied as follows: first, by the definition of $\bar{H}_{t''}$, it is $\bar{H}_{t''}\succeq \Sigma\inv\att{t''}{t''-1}(\calS'\cup\{v\})$, and as a result, Lemma~\ref{lem:inverse} implies $\bar{H}_{t''}\inv\preceq \Sigma\att{t''}{t''-1}(\calS'\cup\{v\})$.  Moreover, Corollary~\ref{cor:from_t_to_t+1} and the fact that $\Sigma\att{1}{1}(\calS\cup \{v\})\preceq \Sigma\att{1}{1}(\emptyset)$, which holds due to Lemma~\ref{prop:one-step_monotonicity}, imply $\Sigma\att{t''}{t''-1}(\calS'\cup\{v\})\preceq \Sigma\att{t''}{t''-1}(\emptyset)$.  Finally, from eq.~\eqref{eq:covariance_riccati} in Lemma~\ref{lem:covariance_riccati} it is~$\Sigma\att{t''}{t''-1}(\emptyset)=\Sigma\att{t''}{t''}(\emptyset)$.  Overall, the desired inequality $\bar{H}_{t''}\inv\preceq \Sigma\att{t''}{t''}(\emptyset)$ holds.
\end{proof}

\section*{Appendix E: Proof of Theorem~\ref{th:freq}}
\label{sec:contribution}

\begin{mylemma}[System-level condition for near-optimal sensor selection]\label{prop:iff_for_all_zero_control}
	Let $N_1$ be defined as in eq.~\eqref{eq:control_riccati}.  The control policy $u_{1:T}\triangleq(0,0,\ldots, 0)$ is suboptimal for the \LQG problem in eq.~\eqref{pr:perfect_state} for all non-zero initial conditions~$x_1$
	if and only if 
	\beq\label{ineq:iff_for_all_zero_control}
	\textstyle \sum_{t=1}^{T} A_1\tran \cdots A_t\tran Q_t A_t\cdots A_1\succ N_1.
	\eeq
\end{mylemma}

\begin{proof}[Proof of Lemma~\ref{prop:iff_for_all_zero_control}]
	For any initial condition $x_1$, eq.~\eqref{eq:lem:opt_sensors} in Lemma~\ref{lem:LQG_closed} implies for the noiseless perfect state information LQG problem in eq.~\eqref{pr:perfect_state}:
	\begin{equation}\label{eq:it_always_aux_1}
	\min_{u_{1:T}}\sum_{t=1}^{T}\left.[\|x\at{t+1}\|^2_{Q_t} +\|u_{t}(x_t)\|^2_{R_t}]\right|_{\Sigma\att{t}{t}=W_t=0}=x_1\tran N_1 x_1,
	\end{equation}
	since $\mathbb{E}(\|x_1\|^2_{N_1})=x_1\tran N_1x_1$, because $x_1$ is known ($\Sigma\att{1}{1}=0$), and $\Sigma\att{t}{t}$ and $W_t$ are zero.
	In addition, for $u_{1:T}=(0,0,\ldots,0)$, the objective function in the noiseless perfect state information LQG problem in eq.~\eqref{pr:perfect_state} is 
	\beal\label{eq:it_always_aux_2}
	&\sum_{t=1}^{T}\left.[\|x\at{t+1}\|^2_{Q_t} +\|u_{t}(x_t)\|^2_{R_t}]\right|_{\Sigma\att{t}{t}=W_t=0}\\
	&= \sum_{t=1}^{T} x\at{t+1}\tran Q_t x\at{t+1}\\ 
	&=x_1\tran\sum_{t=1}^{T} A_1\tran A_2\tran \cdots A_t\tran Q_t A_tA_{t-1}\cdots A_1 x_1,
	\eeal
	since $x_{t+1}=A_tx_t=A_tA_{t-1}x_{t-1}=\ldots=A_tA_{t-1}\cdots A_1x_1$ when all $u_1, u_2, \ldots, u_T$ are zero.  
	
	From eqs.~\eqref{eq:it_always_aux_1} and~\eqref{eq:it_always_aux_2}, the inequality $$x_1\tran N_1 x_1 < x_1\tran\sum_{t=1}^{T} A_1\tran A_2\tran \cdots A_t\tran Q_t A_tA_{t-1}\cdots A_1 x_1$$ holds for any non-zero $x_1$ if and only if
	\belowdisplayskip =-12pt\begin{equation*}
	N_1 \prec \sum_{t=1}^{T} A_1\tran \cdots A_t\tran Q_t A_tA_{t-1}\cdots A_1.
	\end{equation*}
\end{proof}

\begin{mylemma}\label{lem:theta_formula}
For any $\allT$, $$\Theta_t=A_t\tran S_t A_t+Q_{t-1}-S_{t-1}.$$
\end{mylemma}
\begin{proof}[Proof of Lemma~\ref{lem:theta_formula}]
Using the Woobury identity in Lemma~\ref{lem:woodbury}, and the notation in eq.~\eqref{eq:control_riccati},
\begin{align*}
N_t &= A_t\tran (S_t\inv+B_tR_t\inv B_t\tran)\inv A_t \\
&= A_t\tran (S_t - S_t B_t M_t\inv B_t\tran S_t) A_t\\
&=A_t\tran S_t A_t-\Theta_t.
\end{align*}
The latter, gives $\Theta_t=A_t\tran S_t A_t-N_t$.  In addition, from eq.~\eqref{eq:control_riccati},  $-N_t=Q_{t-1}-S_{t-1}$, since $S_t=Q_t+N_{t+1}$.
\end{proof}

\begin{mylemma}\label{lem:observability_condition}
$\sum_{t=1}^{T} A_1\tran A_2\tran \cdots A_t\tran Q_t A_tA_{t-1}\cdots A_1\succ N_1$ if and only if 
\begin{equation*}
\sum_{t=1}^{T}A\tran_{1}A\tran_2\cdots A\tran_{t-1} \Theta_t A_{t-1}A_{t-2}\cdots A_1 \succ 0.
\end{equation*}
\end{mylemma}
\begin{proof}[Proof of Lemma~\ref{lem:observability_condition}]
For $i= t-1, t-2, \ldots, 1$, we pre- and post-multiply the identity in Lemma~\ref{lem:theta_formula} with $A_i\tran$ and $A_i$, respectively:
\beal
&\Theta_t=A_t\tran S_t A_t+Q_{t-1}-S_{t-1}\Rightarrow\\
& A_{t-1}\tran\Theta_t A_{t-1}=A_{t-1}\tran A_t\tran S_t A_t A_{t-1}+A_{t-1}\tran Q_{t-1}A_{t-1}-\\
& \qquad A_{t-1}\tran S_{t-1}A_{t-1}\Rightarrow\\
& A_{t-1}\tran\Theta_t A_{t-1}=A_{t-1}\tran A_t\tran S_t A_t A_{t-1}+A_{t-1}\tran Q_{t-1}A_{t-1}-\\
& \qquad \Theta_{t-1}+Q_{t-2}-S_{t-2}\Rightarrow\\
& \Theta_{t-1}+ A_{t-1}\tran\Theta_t A_{t-1}=A_{t-1}\tran A_t\tran S_t A_t A_{t-1}+\\
&\qquad A_{t-1}\tran Q_{t-1}A_{t-1}+Q_{t-2}-S_{t-2}\Rightarrow\\
&\ldots\Rightarrow\\
& \Theta_2+ A_2\tran\Theta_3 A_2+\ldots+A_2\tran \cdots A_{t-1}\tran\Theta_t A_{t-1}\cdots A_2=\\
& A_2\tran \cdots A_t\tran S_t A_t \cdots A_2+A_2\tran \cdots A_{t-1}\tran Q_{t-1}A_{t-1}\cdots A_2+\\
&\qquad \ldots+ A_2\tran Q_2 A_2+ Q_{1}-S_{1}\Rightarrow\\
& \Theta_1+ A_1\tran\Theta_2 A_1+\ldots+A_1\tran \cdots A_{t-1}\tran\Theta_t A_{t-1}\cdots A_1=\\
& A_1\tran \cdots A_t\tran S_t A_t \cdots A_1+A_1\tran \cdots A_{t-1}\tran Q_{t-1}A_{t-1}\cdots A_1+\\
&\qquad\ldots+ A_1\tran Q_1 A_1-N_1\Rightarrow\\
&\sum_{t=1}^{T}A\tran_{1}\cdots A\tran_{t-1} \Theta_t A_{t-1}\cdots A_1 =\\
&\qquad\sum_{t=1}^{T} A_1\tran \cdots A_t\tran Q_t A_t\cdots A_1-N_1.\label{lem:obser_aux_1}
\eeal
The last equality in eq.~\eqref{lem:obser_aux_1} implies Lemma~\ref{lem:observability_condition}.
\end{proof}

\begin{mylemma}\label{lem:sum_theta}
Consider for any $\allT$ that $A_t$ is invertible.
$\sum_{t=1}^{T}A\tran_{1}A\tran_2\cdots A\tran_{t-1} \Theta_t A_{t-1}A_{t-2}\cdots A_1\succ 0$ if and only if  
$$\sum_{t=1}^{T} \Theta_t \succ 0.$$
\end{mylemma}
\begin{proof}[Proof of Lemma~\ref{lem:sum_theta}] Let $U_t= A_{t-1}A_{t-2}\cdots A_1$.

We first prove that for any non-zero vector $z$, if it is $\sum_{t=1}^{T}A\tran_{1}A\tran_2\cdots A\tran_{t-1} \Theta_t A_{t-1}A_{t-2}\cdots A_1\succ 0$, then $\sum_{t=1}^{T} z\tran\Theta_t z> 0$.  In particular, since $U_t$ is invertible, ---because for any $t\in \{1,2,\ldots, T\}$, $A_t$ is,---
\beal\label{eq:lem_sum_theta_aux_1}
\sum_{t=1}^{T} z\tran\Theta_t z&=\sum_{t=1}^{T} z\tran U_t^{-\top}U_t\tran\Theta_t U_t U_t\inv z\\
&=\sum_{t=1}^T\trace{\phi_t \phi_t\tran U_t\tran\Theta_t U_t},
\eeal
where we let $\phi_t=U_t\inv z$.  Consider a time $t'$ such that for any time $t\in \{1,2\ldots, T\}$, $\phi_{t'} \phi_{t'}\tran\preceq \phi_t \phi_t\tran$.  From eq.~\eqref{eq:lem_sum_theta_aux_1}, using Lemmata~\ref{lem:traceAB_mon} and~\ref{lem:trace_low_bound_lambda_min},
\begin{align*}
\sum_{t=1}^{T} z\tran\Theta_t z&\geq\sum_{t=1}^T\trace{\phi_{t'} \phi_{t'}\tran U_t\tran\Theta_t U_t}\\
&\geq\trace{\phi_{t'} \phi_{t'}\tran \sum_{t=1}^TU_t\tran\Theta_t U_t}\\
&\geq\trace{\phi_{t'} \phi_{t'}\tran}\lambda_\min(\sum_{t=1}^TU_t\tran\Theta_t U_t)\\
&=\|\phi_{t'}\|_2^2 \lambda_\min(\sum_{t=1}^TU_t\tran\Theta_t U_t)\\
&>0.
\end{align*}

We finally prove that for any non-zero vector $z$, if  $\sum_{t=1}^{T} \Theta_t \succ 0$, then $\sum_{t=1}^{T}z A\tran_{1}\cdots A\tran_{t-1} \Theta_t A_{t-1}\cdots A_1z\succ 0$.  In particular,
\begin{align}\label{eq:lem_sum_theta_aux_2}
\sum_{t=1}^{T} z\tran U_t\tran\Theta_t U_t z&=\sum_{t=1}^T\trace{ \xi_{t}\tran \Theta_t \xi_{t}},
\end{align}
where we let $\xi_t=U_tz$.  Consider time $t'$ such that for any time $t\in \{1,2\ldots, T\}$, $\xi_{t'} \xi_{t'}\tran\preceq \xi_t \xi_t\tran$.  From eq.~\eqref{eq:lem_sum_theta_aux_1}, using Lemmata~\ref{lem:traceAB_mon} and~\ref{lem:trace_low_bound_lambda_min},
\belowdisplayskip =-12pt\begin{align*}
\sum_{t=1}^T\trace{ \xi_{t}\tran \Theta_t \xi_{t}}&\geq\trace{\xi_{t'} \xi_{t'}\tran \sum_{t=1}^T\Theta_t }\\
&\geq\trace{\xi_{t'} \xi_{t'}\tran}\lambda_\min(\sum_{t=1}^T\Theta_t)\\
&=\|\xi_{t'}\|_2^2 \lambda_\min(\sum_{t=1}^T\Theta_t )\\
&>0.
\end{align*}
\end{proof}

\begin{proof}[Proof of Theorem~\ref{th:freq}]
Theorem~\ref{th:freq} follows from the sequential application of Lemmata~\ref{prop:iff_for_all_zero_control},~\ref{lem:observability_condition}, and~\ref{lem:sum_theta}.
\end{proof}

\bibliographystyle{IEEEtran}
\bibliography{references}

\begin{thebibliography}{10}
\providecommand{\url}[1]{#1}
\csname url@samestyle\endcsname
\providecommand{\newblock}{\relax}
\providecommand{\bibinfo}[2]{#2}
\providecommand{\BIBentrySTDinterwordspacing}{\spaceskip=0pt\relax}
\providecommand{\BIBentryALTinterwordstretchfactor}{4}
\providecommand{\BIBentryALTinterwordspacing}{\spaceskip=\fontdimen2\font plus
\BIBentryALTinterwordstretchfactor\fontdimen3\font minus
  \fontdimen4\font\relax}
\providecommand{\BIBforeignlanguage}[2]{{%
\expandafter\ifx\csname l@#1\endcsname\relax
\typeout{** WARNING: IEEEtran.bst: No hyphenation pattern has been}%
\typeout{** loaded for the language `#1'. Using the pattern for}%
\typeout{** the default language instead.}%
\else
\language=\csname l@#1\endcsname
\fi
#2}}
\providecommand{\BIBdecl}{\relax}
\BIBdecl

\bibitem{Elia01tac-limitedInfoControl}
N.~Elia and S.~Mitter, ``Stabilization of linear systems with limited
  information,'' \emph{IEEE Trans. on Automatic Control}, vol.~46, no.~9, pp.
  1384--1400, 2001.

\bibitem{Nair04sicon-rateConstrainedControl}
G.~Nair and R.~Evans, ``Stabilizability of stochastic linear systems with
  finite feedback data rates,'' \emph{SIAM Journal on Control and
  Optimization}, vol.~43, no.~2, pp. 413--436, 2004.

\bibitem{Tatikonda04tac-limitedCommControl}
S.~Tatikonda and S.~Mitter, ``Control under communication constraints,''
  \emph{IEEE Trans. on Automatic Control}, vol.~49, no.~7, pp. 1056--1068,
  2004.

\bibitem{Borkar97-limitedCommControl}
V.~Borkar and S.~Mitter, ``{LQG} control with communication constraints,''
  \emph{Comm., Comp., Control, and Signal Processing}, pp. 365--373, 1997.

\bibitem{LeNy14tac-limitedCommControl}
J.~L. Ny and G.~Pappas, ``Differentially private filtering,'' \emph{IEEE Trans.
  on Automatic Control}, vol.~59, no.~2, pp. 341--354, 2014.

\bibitem{Nair07ieee-rateConstrainedControl}
G.~Nair, F.~Fagnani, S.~Zampieri, and R.~Evans, ``Feedback control under data
  rate constraints: An overview,'' \emph{Proceedings of the IEEE}, vol.~95,
  no.~1, pp. 108--137, 2007.

\bibitem{Hespanha07ieee-networkedControl}
J.~Hespanha, P.~Naghshtabrizi, and Y.Xu, ``A survey of results in net- worked
  control systems,'' \emph{Pr.~of the IEEE}, vol.~95, no.~1, p. 138, 2007.

\bibitem{Baillieul07ieee-networkedControl}
J.~Baillieul and P.~Antsaklis, ``Control and communication challenges in
  networked real-time systems,'' \emph{Proceedings of the IEEE}, vol.~95,
  no.~1, pp. 9--28, 2007.

\bibitem{gupta2006stochastic}
V.~Gupta, T.~H. Chung, B.~Hassibi, and R.~M. Murray, ``On a stochastic sensor
  selection algorithm with applications in sensor scheduling and sensor
  coverage,'' \emph{Automatica}, vol.~42, no.~2, pp. 251--260, 2006.

\bibitem{avron2013faster}
H.~Avron and C.~Boutsidis, ``Faster subset selection for matrices and
  applications,'' \emph{SIAM Journal on Matrix Analysis and Applications},
  vol.~34, no.~4, pp. 1464--1499, 2013.

\bibitem{li2017dual}
C.~{Li}, S.~{Jegelka}, and S.~{Sra}, ``{Polynomial Time Algorithms for Dual
  Volume Sampling},'' \emph{ArXiv e-prints: 1703.02674}, 2017.

\bibitem{joshi2009sensor}
S.~Joshi and S.~Boyd, ``Sensor selection via convex optimization,'' \emph{IEEE
  Transactions on Signal Processing}, vol.~57, no.~2, pp. 451--462, 2009.

\bibitem{leny2011kalman}
J.~L. Ny, E.~Feron, and M.~A. Dahleh, ``Scheduling continuous-time kalman
  filters,'' \emph{IEEE Trans. on Aut. Control}, vol.~56, no.~6, pp.
  1381--1394, 2011.

\bibitem{shamaiah2010greedy}
M.~Shamaiah, S.~Banerjee, and H.~Vikalo, ``Greedy sensor selection: Leveraging
  submodularity,'' in \emph{Proceedings of the 49th IEEE Conference on Decision
  and control}, 2010, pp. 2572--2577.

\bibitem{jawaid2015submodularity}
S.~T. Jawaid and S.~L. Smith, ``Submodularity and greedy algorithms in sensor
  scheduling for linear dynamical systems,'' \emph{Automatica}, vol.~61, pp.
  282--288, 2015.

\bibitem{tzoumas2015sensor}
V.~Tzoumas, A.~Jadbabaie, and G.~J. Pappas, ``Sensor placement for optimal
  {K}alman filtering,'' in \emph{Amer.~Contr.~Conf.}, 2016, pp. 191--196.

\bibitem{Shafieepoorfard13cdc-attentionLQG}
E.~Shafieepoorfard and M.~Raginsky, ``Rational inattention in scalar {LQG}
  control,'' in \emph{Proceedings of the 52th IEEE Conference in Decision and
  Control}, 2013, pp. 5733--5739.

\bibitem{tanaka2015sdp}
T.~Tanaka and H.~Sandberg, ``{SDP}-based joint sensor and controller design for
  information-regularized optimal {LQG} control,'' in \emph{54th IEEE
  Conference on Decision and Control}, 2015, pp. 4486--4491.

\bibitem{bertsekas2005dynamic}
D.~P. Bertsekas, \emph{Dynamic programming and optimal control,
  Vol.~{I}}.\hskip 1em plus 0.5em minus 0.4em\relax Athena Scientific, 2005.

\bibitem{nemhauser78analysis}
G.~Nemhauser, L.~Wolsey, and M.~Fisher, ``An analysis of approximations for
  maximizing submodular set functions -- {I},'' \emph{Mathematical
  Programming}, vol.~14, no.~1, pp. 265--294, 1978.

\bibitem{wang2016approximation}
Z.~Wang, B.~Moran, X.~Wang, and Q.~Pan, ``Approximation for maximizing monotone
  non-decreasing set functions with a greedy method,'' \emph{Journal of
  Combinatorial Optimization}, vol.~31, no.~1, pp. 29--43, 2016.

\bibitem{feige1998}
U.~Feige, ``A threshold of $ln(n)$ for approximating set cover,'' \emph{Journal
  of the ACM}, vol.~45, no.~4, pp. 634--652, 1998.

\bibitem{carlone2017attention}
L.~Carlone and S.~Karaman, ``Attention and anticipation in fast visual-inertial
  navigation,'' in \emph{Proceeding of the IEEE International Conference on
  Robotics and Automation}, 2017, pp. 3886--3893.

\bibitem{bernstein2005matrix}
D.~S. Bernstein, \emph{Matrix mathematics}.\hskip 1em plus 0.5em minus
  0.4em\relax Princeton University Press, 2005.

\bibitem{chamon2016near}
L.~F. Chamon and A.~Ribeiro, ``Near-optimality of greedy set selection in the
  sampling of graph signals,'' in \emph{Proceedings of the IEEE Global
  Conference on Signal and Information Processing}, 2016, pp. 1265--1269.

\bibitem{coppersmith1990matrix}
D.~Coppersmith and S.~Winograd, ``Matrix multiplication via arithmetic
  progressions,'' \emph{J.~of Symbolic Comp.}, vol.~9, no.~3, pp. 251--280,
  1990.

\end{thebibliography}

\end{document}